\documentclass[a4paper,reqno]{article}

	\usepackage[utf8]{inputenc}
	\usepackage{xspace}

	\usepackage{xcolor}

	\usepackage{enumerate}
	\usepackage[textsize=footnotesize]{todonotes}
		\presetkeys{todonotes}{fancyline, color=blue!30,bordercolor=blue!30}{}

	\usepackage{amssymb,amsmath,amsthm}
	\usepackage{tikz-cd}
	\usepackage{tikz}
		\usetikzlibrary{mindmap}
	\usepackage[ampersand]{easylist}

	\usepackage[abbrev]{amsrefs}
	\usepackage{hyperref}
		\definecolor{fancyblue}{HTML}{0000bb}
		\definecolor{fancyred}{HTML}{c40000}
		\hypersetup{colorlinks=true,linkcolor=fancyblue,citecolor=fancyred}
	\usepackage[nameinlink,capitalise,noabbrev]{cleveref} 
		\Crefname{equation}{equation}{equations}
		\Crefname{page}{page}{pages}

	\newtheorem{theorem}{Theorem}[section]
	\newtheorem{lemma}[theorem]{Lemma}
	\newtheorem{proposition}[theorem]{Proposition}
	\newtheorem{corollary}[theorem]{Corollary}

	\newtheorem*{theorem*}{Theorem}
	\newtheorem*{lemma*}{Lemma}
	\newtheorem*{proposition*}{Proposition}
	\newtheorem*{corollary*}{Corollary}
	\theoremstyle{definition}
	\newtheorem{definition}[theorem]{Definition}
	\newtheorem*{definition*}{Definition}
	\newtheorem{example}[theorem]{Example}
	\newtheorem*{example*}{Example}

	\theoremstyle{remark}
	\newtheorem{remark}[theorem]{Remark}
	\newtheorem*{remark*}{Remark}

\numberwithin{equation}{section}

\providecommand{\TT}{\mathbb{T}} 
\providecommand{\ET}{\mathrm{E}\TT}
\providecommand{\BT}{\mathrm{B}\TT}
\providecommand{\Bor}[1]{{#1}_{\mathrm{h}\TT}}
\providecommand{\Borr}[1]{{\widetilde{#1}}_{\mathrm{h}\TT}}

\DeclareMathOperator{\End}{End}

\providecommand{\tH}{\tilde{H}} 
\providecommand{\BH}{\tilde{H}^{\TT}} 
\providecommand{\BC}{\tilde{H}_{\TT}} 
\providecommand{\tBH}{t\tilde{H}^{\TT}} 
\providecommand{\tBC}{t\tilde{H}_{\TT}} 
\providecommand{\cBH}{c\tilde{H}^{\TT}} 
\providecommand{\cBC}{c\tilde{H}_{\TT}} 

\DeclareMathOperator{\colim}{colim}
\DeclareMathOperator{\Tor}{Tor}

\DeclareMathOperator{\Hom}{Hom}

\DeclareMathOperator{\im}{im}

\providecommand{\SWF}{\mathrm{SWF}}

\providecommand{\spinc}{spin$^c$\xspace}

	\DeclareFontFamily{U}{mathx}{\hyphenchar\font45}
	\DeclareFontShape{U}{mathx}{m}{n}{
	      <5> <6> <7> <8> <9> <10>
	      <10.95> <12> <14.4> <17.28> <20.74> <24.88>
	      mathx10
	      }{}
	\DeclareSymbolFont{mathx}{U}{mathx}{m}{n}
	\DeclareFontSubstitution{U}{mathx}{m}{n}
	\DeclareMathAccent{\widecheck}{0}{mathx}{"71}

\providecommand{\HM}{\mathrm{HM}}
\providecommand{\HMfrom}{\widehat{\HM}}
\providecommand{\HMbar}{\overline{\HM}}
\providecommand{\HMto}{\widecheck{\HM}}

\providecommand{\N}{\mathbb{N}}
\providecommand{\Z}{\mathbb{Z}}
\providecommand{\Q}{\mathbb{Q}}
\providecommand{\R}{\mathbb{R}}
\providecommand{\C}{\mathbb{C}}

\providecommand{\F}{\mathbb{F}} 
\providecommand{\kk}{\Bbbk} 

\newcommand{\homeo}{\approx}

\newcommand{\homequ}{\simeq}

\providecommand{\CP}{\mathbb{C}\mathrm{P}}
\providecommand{\CPinf}{\CP^\infty}

\providecommand{\mr}[1]{\mathrm{#1}}

\providecommand{\mc}[1]{\mathcal{#1}}
\providecommand{\mf}[1]{\mathfrak{#1}}

\providecommand{\ra}{\rightarrow}
\providecommand{\lra}{\longrightarrow}
\providecommand{\hra}{\hookrightarrow}

\providecommand{\lrao}[1]{\overset{#1}{\lra}}
\providecommand{\xra}[1]{\xrightarrow{#1}}

\providecommand{\Set}[2]{ \left\{ #1 \,\middle|\, #2 \right\} }

\providecommand{\scp}[1]{\langle #1 \rangle}

\providecommand{\inv}{^{-1}}

\DeclareMathOperator{\ind}{ind}
\DeclareMathOperator{\id}{id}

\providecommand{\eqand}{\quad\text{and}\quad}
\providecommand{\AND}{\quad\text{and}\quad}

\providecommand{\wrt}{with respect to\xspace}

\providecommand{\Froyshov}{Fr\o{}yshov\xspace}

\providecommand{\ETwt}{\ET_+\wedge_\TT}

\providecommand{\Hk}{\mr H\kk}

	\title{The monopole $h$--invariants\\ from a topological perspective}
	\author{Stefan Behrens%
		\thanks{Bielefeld University,
		Faculty of Mathematics,
		Universitätsstr.~25, 
		D--33615 Bielefeld, 
		Germany; 
		\textit{Email address:} \texttt{sbehrens@math.uni-bielefeld.de}
		}
	}
\date{}

\begin{document}

	\maketitle

\begin{abstract}
\noindent
We study the monopole $h$--invariants of 3--manifolds from a topological perspective based on Lidman and Manolescu's description of monopole Floer homology in terms of Seiberg--Witten--Floer homotopy types.
We investigate the possible dependence on the choice of coefficients and give proofs of several properties of the $h$--invariants which are well known to experts, but hard to track down in the literature.
\end{abstract}

\section{Introduction}

This article lives in two worlds.
The methods come from equivariant algebraic topology, but the motivation comes from Seiberg--Witten theory on $3$--manifolds.
The two worlds are connected by the work of Manolescu~\cites{Manolescu_SWF_spectra_2003,Manolescu_gluing_theorem_2007,Manolescu_spin_with_boundary_2014}.
We explore the monopole $h$--invariants, certain rational numbers defined using Kronheimer and Mrowka's monopole Floer homology~\cite{KronheimerMrowka_book_2007}, from the topological perspective using the results of Lidman and Manolescu~\cite{LidmanManolescu_equivalance_2018}.
Along the way, we prove several properties of the $h$--invariants which are well known to experts, but hard to track down in the literature.
We tried to keep the exposition accessible to readers from both worlds which do not necessarily live in the border region.

Throughout, we use the following notation:
$\TT$ is the circle group, 
$\kk$ is a principal ideal domain, 
$\F$ is a field, and 
$\F_p$ is a prime field%
of characteristic~$p$ which may be zero or a prime number.
Concretely, we may identify~$\F_0=\Q$ and~$\F_p=\Z/p\Z$ for~$p>0$.
The letter~$X$ will be reserved for~$\TT$--spaces.
Similarly, we reserve~$Y$ for 3--manifolds that are closed, connected, and oriented by default.

\paragraph{Motivation from Seiberg--Witten theory.}

Let~$Y$ be a closed, connected, oriented 3--manifold, implicitly equipped with a Riemannian metric and a \spinc structure.
Monopole Floer homology as defined by Kronheimer and Mrowka~\cite{KronheimerMrowka_book_2007} assigns to this data a long exact sequence 
	\begin{equation*}\begin{tikzcd}[column sep=small]
	\cdots \rar&
		\HMfrom_{*+1}(Y;\kk) \rar{p_*^\kk} & 
			\HMbar_{*}(Y;\kk) \rar{i_*^\kk} & 
				\HMto_{*}(Y;\kk) \rar{j_*^\kk} & 
					\HMfrom_{*}(Y;\kk) \rar &
						\cdots
	\end{tikzcd}\end{equation*}
of $\kk[u]$--modules which are relatively $\Z$--graded as $\kk$--modules such that the action of~$u$ has degree~$-2$.
Further, the $u$--action is invertible on~$\HMbar_*(Y;\kk)$.

If~$b_1(Y)=0$, then the relative grading lifts to an absolute grading with values in a subset of~$\Q$ of the form~$\mf q(Y)=\Z-2n(Y)$ where~$n(Y)\in\Q$ is a spectral invariant of a \spinc Dirac operator that arises in the construction.
Moreover, it is known that~$\HMbar_{*}(Y;\kk)$ is supported in degrees~$\mf q^\mathrm{ev}(Y)=2\Z-2n(Y)$ in which it is isomorphic to~$\kk$.
The following definition, which is a variation of~\cite{KronheimerMrowkaOzsvathSzabo_lens_space_surgeries_2007}*{Def.~2.11}, turns out to be sensible:
\begin{definition}\label{D:monopole h-invariants homology intro}
Let~$Y$ be a \spinc 3--manifold with~$b_1(Y)=0$.
For every principal ideal domain~$\kk$ we define the \emph{weak} and \emph{strong} \emph{monopole $h$--invariants}
	\begin{align*}
	h_w(Y;\kk) &= \tfrac12\max\Set{q\in q^\mathrm{ev}(Y)}{p_{k-1}^\kk\ne0} \\
	h_s(Y;\kk) &= \tfrac12\max\Set{q\in q^\mathrm{ev}(Y)}{\text{$p_{k-1}^\kk$ is surjective}}.
	\end{align*}
If~$\F$ is a field, we write~$h_\F(Y)=h_w(Y;\kk)=h_s(Y;\kk)$ and further abbreviate to~$h_p(Y)=h_{\F_p}(Y)$ for the prime fields~$\F_p$.
\end{definition}
The literature about Floer theory on 3--manifolds contains several similar and closely related definitions; we give an overview in \cref{ch:precursors}.
The significance of the invariants~$h_p(Y)$ comes from the following results, which are common folklore, but hard to track down in the literature.
In \cref{ch:additivity monopoles,ch:Froyshov} we offer new proofs based on the Seiberg--Witten--Floer homotopy types defined by Manolescu in~\cite{Manolescu_SWF_spectra_2003}.

\begin{theorem}\label{T:Froyshov intro}
We have $h_{s/w}(S^3;\kk)=0$ and if $Y$~is the boundary of a compact \spinc 4--manifold~$W$ with negative definite intersection form, then
	\begin{equation}
	\frac18\big(c_1^2(W)+b_2(W)\big) \le h_{s/w}(Y;\kk)
	\hspace{2em}\text{(``\Froyshov inequality'')}
	\end{equation}
where~$c_1(W)$ is considered as element of~$H^2(X,Y;\Q)\cong H^2(X;\Q)$.
\end{theorem}
\begin{proof}
A slightly more general version will be proved in \cref{T:Froyshov main}.
\end{proof}
As is well known, \cref{T:Froyshov intro} imposes restrictions on the possible definite intersection forms on smooth 4--manifolds with boundary which generalize Donaldson's celebrated diagonalizability theorem (cf.~\cite{KronheimerMrowka_book_2007}*{Thm.~39.1.4}, \cite{BehrensGolla_twisted_2018}*{Ch.~5}).
The invariants~$h_p$ defined using prime fields enjoy some additional properties that facilitate computations.
\begin{theorem}\label{T:additivity etc intro}
Let~$Y,Y'$ be \spinc 3--manifolds with~$b_1(Y)=0$.
The following hold for~$p$ equal to zero or a prime number:
\begin{enumerate}[(i)]
\item
$h_p(Y\# Y') = h_p(Y)+h_p(Y')$ \hfill{(``Additivity'')}
\item
$h_p(-Y) = -h_p(Y)$ \hfill{(``Duality'')}
\item
$h_p(Y)+\tfrac18\big(c_1^2(W)+b_2(W)\big)\le h_p(Y')$ where~$W$ is a \spinc cobordism from~$Y$ to~$Y'$ with negative definite intersection form. \hfill{(``Monotonicity'')}
\end{enumerate}
\end{theorem}
\begin{proof}
The three properties will be proved individually in \cref{T:duality,T:additivity,T:monotonicity mon}.
\end{proof}

A priori, the invariants~$h_{s/w}(Y;\kk)$ have no reason to be independent of the choice of~$\kk$ and the obvious inequality~$h_w(Y;\kk)\ge h_s(Y;\kk)$ has no reason not to be strict.
However, as far as we know, no examples are known whose $h$--invariants explicitly depend on the choices.
This issue has been raised before, for example in~\cite{Manolescu_triangulaion_JAMS_2016}*{Rem.~3.12}, and it is pointed out in~\cite{Froyshov_monopole_homology_2010}*{p.~569} that a possible discrepancy~$h_p(Y)\ne h_0(Y)$ is related to the $\Z$--torsion in~$\HMto_*(Y;\Z)$.
One of our main objectives is to improve the understanding of the choice of coefficients and the potential difference in the weak and strong flavors.
To that end, we will prove the following result which emphasizes the role of~$\F_p$ and~$\Z$.

\begin{theorem}\label{T:coeff dep intro}
Let~$Y$ be a \spinc 3--manifold with~$b_1(Y)=0$.
\begin{enumerate}[(i)]
\item
If~$\kk$ be a principal ideal domain of characteristic\footnote{The \emph{characteristic} of a principal ideal domain~$\kk$ is that of its fields of fractions.}~$p$, then
	\begin{equation*}
	h_w(Y;\kk)=h_p(Y) \AND 
	h_s(Y;\kk)\ge h_s(Y;\Z).
	\end{equation*}
\item
The monopole $h$--invariants with integer coefficients are given by
	\begin{equation*}
	h_w(Y;\Z) = h_0(Y) \AND 
	h_s(Y;\Z) = \min_p h_p(Y).
	\end{equation*}
\end{enumerate}
\end{theorem}
\begin{proof}
This will be part of \cref{T:coeff dep monopoles}.
\end{proof}
In summary, the strong $h$--invariant~$h_s(Y;\Z)$ gives the potentially strongest \Froyshov inequalities in \cref{T:Froyshov intro}, and it can be computed using prime field coefficients only, but one has to work with all prime fields.

\paragraph{Connection to $\TT$--equivariant topology.}
As mentioned, we approach the monopole $h$--invariants using the Seiberg--Witten--Floer homotopy types defined in~\cite{Manolescu_SWF_spectra_2003}.
The following definition will be central:  

\begin{definition}[cf.~\cites{Manolescu_spin_with_boundary_2014,Manolescu_triangulaion_JAMS_2016,DaiSasahiraStoffregen_lattice_vs_SWF_arxiv_v1_2023}]\label{D:SWF spaces intro}
A $\TT$--space~$X$ is said to be of \emph{type SWF} if the following conditions are satisfied:
\begin{enumerate}[(a)]
\item
$X$ is $\TT$--homotopy equivalent to a finite $\TT$--complex.
\item 
$\TT$ acts freely on~$X\setminus X^\TT$.
\item
$X^\TT\simeq S^\ell$ for some~$\ell$.
\end{enumerate}
\end{definition}

To every \spinc 3--manifold~$Y$ with~$b_1	(Y)=0$, Manolescu~\cite{Manolescu_SWF_spectra_2003} associates a $\TT$--equivariant stable homotopy type~$\SWF(Y)$ which can be represented by triples~$(X,\ell,n)$ where~$X$ is a $\TT$--space of type SWF, $\ell\in\Z$ is given by~$X^\TT\simeq S^\ell$, and $n\in \Z-n(Y)\subset\Q$.
Intuitively, $(X,\ell,n)$ should be thought of as a `desuspension' $\Sigma^{-\ell\R}\Sigma^{-n\C} X$, with the caveat that~$n$ is usually not an integer.
We refer to~\cite{DaiSasahiraStoffregen_lattice_vs_SWF_arxiv_v1_2023}*{Ch.3} for more details and a recent survey of the theory.
By a theorem of Lidman and Manolescu~\cite{LidmanManolescu_equivalance_2018}*{Thm.~1.2.1} (cf.~\cref{T:Lidman-Manolescu cohomology}), there is an isomorphism of~$\Z[u]$--modules 
	\begin{equation}\label{eq:Lidman-Manolescu intro}
	\HMto_*(Y;\Z)\cong \BH_{*+\ell+2n}(X;\Z).
	\end{equation}
Taking inspiration from \cite{Manolescu_triangulaion_JAMS_2016}*{Ch.~2.6}, we make the following definition:

\begin{definition}\label{D:topological h-invariants intro}
Let~$X$ be a $\TT$--space of type SWF with~$X^\TT\simeq S^\ell$. 
We write $i_X\colon X^\TT\hra X$ for the inclusion and define the \emph{weak} and \emph{strong} \emph{cohomological h-invariants} of~$X$ as
	\begin{align*}
	h^s(X;\Bbbk) &= \min\Set{k\ge0}{\text{$\BC^{\ell+2k}(X;\Bbbk) \xra{i_X^*} \BC^{\ell+2k}(X^\TT;\Bbbk)$ is surjective}}\\
	h^w(X;\Bbbk) &= \min\Set{k\ge0}{\text{$\BC^{\ell+2k}(X;\Bbbk) \xra{i_X^*} \BC^{\ell+2k}(X^\TT;\Bbbk)$ is non-zero}}.
	\end{align*}
Again, for fields we write~$h^\F(X) = h^{s/w}(X;\F)$ and~$h^p(X) = h^{\F_p}(X)$.
\end{definition}

Note that \cref{D:topological h-invariants intro} uses cohomology groups while the isomorphism in~\eqref{eq:Lidman-Manolescu intro} and \cref{D:monopole h-invariants homology intro} use homology groups.
This is reconciled by several layers of duality which are discussed in \cref{ch:HM review,ch:additivity etc,ch:relation to SW} which involve homological and cohomological versions of both the monopole $h$--invariants and their topological analogues.
For the moment, we content ourselves with the following statement:

\begin{proposition}\label{T:comparing h-invariants intro}
Let~$Y$ be a closed \spinc 3--manifold with~$b_1(Y)$ and~$\kk$ a principal ideal domain.
If~$\SWF(-Y)$ is represented by~$(X^*,\ell^*,n^*)$, then
	\begin{equation}\label{eq:h-invariant comparison homology intro}
	h_{s/w}(Y;\kk) 
	= n^*-h^{s/w}(X^*;\kk).
	\end{equation}
\end{proposition}
\begin{proof}
This follows from \cref{T:monopole h-inv hom vs coh,T:comparing h-invariants cohomology}.
\end{proof}
\cref{T:Froyshov intro,T:additivity etc intro,T:coeff dep intro} will be deduced from analogues for the topological \mbox{$h$--invariants} which will be stated and proved in~\cref{ch:topological h-invariants,ch:additivity etc}, respectively.
As for the dependence on coefficients, we will obtain the following in \cref{ch:examples}:
\begin{theorem}\label{T:examples intro}
There are $\TT$--spaces of type SWF~$X$ such that~$h^p(X)\ne h^q(X)$
\end{theorem}
\begin{proof}
This follows from \cref{T:h-invariants of X_ab} and \cref{eg:example for h-difference}.
\end{proof}
Unfortunately, we are not aware of any such examples arising from~$SWF(Y)$ of a 3--manifold~$Y$, leaving the problem of a possible coefficient dependence of the monopole $h$--invariants open.

\begin{remark}\label{R:erratum etc}
(i)
The absence of~$h_s(Y;\Z)$ from \cref{T:additivity etc intro} should also be noted.
In fact, the duality property~$h_s(Y;\Z) = -h_s(-Y;\Z)$ is equivalent to~$h_p(Y)$ being independent of~$p$.
One direction follows from the description of $h_s(Y;\Z)$ in \cref{T:coeff dep intro}(ii), which together with \cref{T:additivity etc intro}(ii) also gives $h_s(-Y;\Z)=-\max_p h_p(Y)$.
For the other direction, we note that \cref{T:additivity etc intro}(ii) and \cref{T:coeff dep intro}(i) applied to~$Y$ and~$-Y$ give
	\begin{equation}
	h_s(Y;\Z)\le h_p(Y) = -h_p(Y) \le -h_s(-Y;\Z).
	\end{equation}

(ii)
The work on the present article was partly motivated by a careless error in the author's joint article~\cite{BehrensGolla_twisted_2018} with Marco Golla.
In~\cite{BehrensGolla_twisted_2018}*{Prop.~3.4} we claimed an inequality $\underline{d}_0(Y,\mf t)\ge \underline{d}_p(Y,\mf t)$ which would imply~$h_0(Y)\ge h_p(Y)$ for all~$p$ (see \cref{ch:precursors}).
In turn, this would imply that~$h_0(Y)=h_p(Y)$ for all~$p$ by a similar argument as in~(i).
While we are not aware of any counterexamples, it has come to our attention that certain abstract chain complexes constructed in~\cite{GorskyLidmanBeibeiMoore_triple_linking_2022} contradict the argument in the proof of~\cite{BehrensGolla_twisted_2018}*{Prop.~3.4}.
\cref{T:coeff dep intro} resulted from an attempt to rectify the argument, while \cref{T:examples intro} 
provides further evidence that~$h_p(Y)$ should not be independent of~$p$, in general.

The main results of~\cite{BehrensGolla_twisted_2018} are unaffected and we are not aware of any references to the potentially faulty statement. 
The journal has been made aware of the situation and the ar$\chi$iv version of~\cite{BehrensGolla_twisted_2018} will be updated.
\end{remark}

\paragraph{Acknowledgments.}
The author would like to thank Marco Golla by pointing out the error in~\cite{BehrensGolla_twisted_2018} outlined in \cref{R:erratum etc}(ii) above.

\section{Circle actions and Borel cohomology}
\label{ch:Borel}

Let~$\TT$ denote the multiplicative group complex numbers of unit modulus.
We use the terminology $\TT$--spaces, $\TT$--maps, etc.\ for spaces\footnote{We implicitly work in a convenient category of spaces, e.g.~compactly generated weak Hausdorff spaces.} with continuous left $\TT$--actions, $\TT$--equivariant maps, and so on.
Base points for $\TT$--spaces are required to be \mbox{$\TT$--fixed} 
and non-degenerate\footnote{i.e.\ the inclusion $\{*\}\hra X$ is a cofibration}.
The notation~$X_+$ indicates the adjunction of a disjoint base point.

By a \emph{$\TT$--representation} we mean a real inner product space~$V$ of finite dimension~$|V|$ together with a $\TT$--action by orthogonal transformations.
We write~$S(V)$, $D(V)$, and~$S^V$ for the unit sphere, unit disk, and one-point compactification of~$V$.
The latter is based at infinity.
We implicitly fix a based $\TT$--homeomorphism $D(V)_+/S(V)_+\cong S^V$ to obtain cofiber sequence of based $\TT$--spaces
	\begin{equation}\label{eq:cofiber seq V}
	S(V)_+\to D(V)_+ \to S^V.
	\end{equation}
For a based $\TT$--space~$X$ we define its \emph{$V$--suspension} as~$\Sigma^VX=X\wedge S^V$.
We will mostly deal with representations of the form~$\R^k\oplus\C^k$ with $\TT$--acting trivially on~$\R$ and by complex multiplication on~$\C$, and we write $S^{k,m}$ for the one-point compactification.

Let~$\C^\infty=\colim_n \C^n$ be the direct sum of countably many copies of~$\C$ with the indicated colimit topology.
Then~$\ET=S(\C^\infty)=\colim_nS(\C^n)$ is a non-equivariantly contractible CW~complex on which $\TT$--acts freely with orbit space~$\BT=\ET/\TT\homeo \CPinf$.
For a $\TT$--space~$X$ the \emph{unreduced Borel construction}~$\Bor{X}=\ET\times_\TT X$ is the total space of a fiber bundle with typical fiber~$X$ over~$\BT$.
We have a fiber sequence of unbased spaces
	\begin{equation}
	X\overset{j}{\hra} \Bor X \xra p \BT
	\end{equation}
with $j$ defined by~$j(x)=[e,x]$ with some~$e\in \ET$ chosen once and for all.
If~$X$ is based, the base point~$*$ gives rise to a canonical section $s\colon \BT\to \Bor X$ with image~$\Bor *$.
Since~$*$ is non-degenerate, $s$ is a cofibration and its cofiber is the \emph{reduced Borel construction}
	$\Borr X = \ET_+\wedge_\TT X \cong \Bor X/\Bor *$.
We record the cofiber sequence as
	\begin{equation}
	\BT \xra s \Bor X \xra q \Borr X.
	\end{equation}
The reduced and unreduced \emph{Borel cohomology} groups of~$X$ with coefficients in a commutative ring with unit~$\kk$ are defined as the ordinary cohomology groups
	\begin{equation}\label{eq:Borel cohomology def}\begin{split}
	H^*_\TT(X;\kk) &= H^*(\Bor X;\kk) \cong \tilde{H}^*\big((\Bor X)_+;\kk\big) \\
	\tilde{H}^*_\TT(X;\kk) &= \tilde{H}^*(\Borr X;\kk) \cong H^*(\Bor X,\Bor *;\kk).
	\end{split}\end{equation}
The \emph{Borel homology} groups~$H_*^\TT(X;\kk)$ and $H_*^\TT(X;\kk)$ are defined analogously.
We omit the coefficients from the notation when no confusion can arise.
We will later make use of the following facts about Borel (co-)homology which can be found in the textbooks~\cites{tomDieck_transformation_groups_1987,Hsiang_Borel_cohomology_book_1975}:
\begin{enumerate}[({B}1)]
\item \label{fact:module structure}
$\BC^*(X)$ and~$\BH_*(X)$ as well as their unreduced analogues are modules over the polynomial ring $\BC^*(S^0)\cong H^*(\CPinf)=\kk[u]$ with~$u\in H^2(\CP^\infty)$ the image of the Euler class of the tautological line bundle under the change of coefficients map~$H^2(\CPinf;\Z)\to H^2(\CPinf;\kk)$.

\item \label{fact: suspension isomorphisms}
If~$X$ is a $\TT$--space and~$V$ a $\TT$--representation, then $\Bor{(X\times V)}$ is a vector bundle over~$\Bor X$.
Since $\TT$ is connected, this bundle is orientable and its orientations correspond to those of~$V$.
If~$X$ is based, then~$\Borr{(X\wedge S^V)}$ is the obtained from the Thom space of~$\Bor{(X\times V)}\to \Bor X$ by collapsing that of~$\Bor{(*\times V)} \to \Bor *$.
After fixing an orientation on~$V$, the corresponding (relative) Thom isomorphism gives  \emph{suspension isomorphisms}
	\begin{equation}
	s_V\colon \BC^*(X)\xra\cong \BC^{*+|V|}(X\wedge S^V).
	\end{equation}
for all coefficient rings~$\kk$.

\item \label{fact: trivial action}
If~$\TT$ acts trivially on~$X$,
then~$\Bor X \homeo \BT\times X$ and~$\Borr X\homeo \BT_+\wedge X$.
In particular, since~$\BT\homeo\CPinf$ has torsion free integral cohomology, the non-equivariant Künneth theorem gives isomorphisms
	\begin{equation}
	H^*_\TT(X)\cong H^*(\BT)\otimes_\Z H^*(X) \AND 
	\BC^*(X)\cong H^*(\BT)\otimes_\Z \tH^*(X)
	\end{equation}
and the (left) $H^*(\BT)$--module structures on the left correspond to the obvious ones on the right.

\item \label{fact: free action}
If~$\TT$ acts freely on~$X$, then the canonical map $\Bor X\to X/\TT$ is a homotopy equivalence.
Similarly, if~$X$ is based and $\TT$--acts freely on~$X\setminus\{*\}$,
then $\Bor X\to X/\TT$ factors through a homotopy equivalence~$\Borr X \homequ X/\TT$.

\item \label{fact: Serre spectral sequence}
Since~$\BT$ is simply-connected and has torsion free integral cohomology, the Serre spectral sequence for~$\Bor X\to \BT$ takes the form 
	\begin{equation}
	E_2^{p,q} \cong H^p(\BT)\otimes_\Z H^q(X) \Longrightarrow H^{p+q}_G(X).
	\end{equation}
The the $E_\infty$~terms along the edges are related to the fibration as follows:
	\begin{equation}
	E_\infty^{0,*}\cong \im j^*\subset H^*(X)
	\AND
	E_\infty^{*,0}\cong \im p^*\subset H^*(\Bor X).
	\end{equation}

\item \label{fact: reduced vs unreduced}
For based~$X$, the maps~$p$, $s$, and~$q$ induce split exact sequences
	\begin{equation*}\begin{tikzcd}
	0 \rar &
		\BC^*(X) \rar{q^*} 	&
			H^*_\TT(X) \rar{s^*} &
				H^*(\BT) \rar \ar[l,"p^*"',bend right=30] &
					0 \\
	0 \rar &
		H_*(\BT) \rar{s_*} &
			H_*^\TT(X) \rar{q_*} \ar[l,"p_*"',bend right=30] &
				\tilde H_*^\TT(X) \rar &
					0
	\end{tikzcd}\end{equation*}
relating the reduced and unreduced versions of Borel (co-)homology.

\item \label{fact: Borel vs colimits}
Reduced Borel homology and cohomology are additive under wedge sums.
More generally, the functor $X\mapsto \Bor X$ preserves wedge sums and push-outs.
In fact, it has a right adjoint~$Z\mapsto F(\ET_+,Z)$ where $Z$ is a space with trivial $\TT$--action and~$F$ is the based function space with $\TT$ acting by conjugation.
Hence, $X\mapsto \Bor X$ preserves all colimits.
\end{enumerate}

\section{The monopole h--invariants}
\label{ch:HM review}

In this section we review the relevant background from Seiberg--Witten theory on 3--manifolds.
Let~$Y$ be a closed, connected, oriented 3--manifold, implicitly equipped with a Riemannian metric~$g$ and a \spinc structure~$\mf t$.
Throughout, we will always assume that~$b_1(Y)=0$.
The \spinc structure determines and is determined by a spinor bundle~$S$ with Clifford multiplication~$\rho\colon T^*Y\to\End_\C(S)$ (see~\cite{KronheimerMrowka_book_2007}*{Ch.~1.1}).
We also fix a \spinc connection~$B_0$ on~$S$.
\paragraph{Monopole Floer homology and cohomology.}
Seiberg--Witten theory on~$Y$ is governed by the \emph{Chern--Simons--Dirac functional}
	\begin{equation}
	\mc L\colon i\Omega^1(Y)\oplus \Gamma(S)\to \R,\quad \mc L(b,\phi)=\tfrac12\scp{b,*db}_{L^2}+\tfrac12\scp{\phi,D\phi+\rho(b)}_{L^2}.
	\end{equation}
in the sense that the downward gradient flow equation $\dot x + \nabla\mc L(x)=0$ recovers the Seiberg--Witten equations on the cylinder~$\R\times Y$.
In the case~$b_1(Y)=0$, the Chern--Simons--Dirac functional is invariant under the action of the group of smooth maps~$u\colon Y\to\TT$ sending a configuration~$(b,\phi)$ to~$(b-u\inv du,u\phi)$.
The subgroup of constant maps is canonically identified with~$\TT$.
In a nutshell, \emph{monopole Floer (co-)homology} as defined by Kronheimer and Mrowka in~\cite{KronheimerMrowka_book_2007} is obtained by adapting a construction from finite dimensional $\TT$--equivariant Morse theory to suitable perturbations of~$\mc L$.
The outcome is a long exact sequence of $\kk[u]$--modules
	\begin{equation}\label{eq:monopole exact sequence homology}\begin{tikzcd}[column sep=small]
	\cdots \rar&
		\HMfrom_{*+1}(Y) \rar{p_*} & 
			\HMbar_{*}(Y) \rar{i_*} & 
				\HMto_{*}(Y) \rar{j_*} & 
					\HMfrom_{*}(Y) \rar &
						\cdots
	\end{tikzcd}\end{equation}
which are relatively $\Z$--graded over~$\kk$ and the action of~$u$ has degree~$-2$.
There is also a cohomological version
	\begin{equation}\label{eq:monopole exact sequence cohomology}\begin{tikzcd}[column sep=small]
	\cdots &
		\HMfrom^{*+1}(Y) \ar[l] & 
			\HMbar^{*}(Y) \ar[l,"p^*"'] & 
				\HMto^{*}(Y) \ar[l,"i^*"'] & 
					\HMfrom^{*}(Y) \ar[l,"j^*"'] &
						\cdots \ar[l]
	\end{tikzcd}\end{equation}
where the $u$--action has degree~$+2$.
The following algebraic features are known:
\begin{enumerate}[({HM}1)]
\item
$\HMfrom_*(Y)$ and~$\HMto^*(Y)$ are finitely generated over~$\kk[u]$.
\item
The $u$--actions on~$\HMbar_*(Y)$ and~$\HMbar^*(Y)$ are invertible.
\item
The maps $p_*$ and~$i^*$ are exhibit $\HMbar_*(Y)$ and~$\HMbar^*(Y)$ as the localizations $u\inv\HMfrom_*(Y)$ and~$u\inv\HMto^*(Y)$.
\item
The relative $\Z$--gradings can be lifted to absolute gradings with values in $\mf q(Y) = \Z - 2n(Y)\subset\Q$ where $n(Y)$ is a spectral invariant associated to the Dirac operator~$D$ (cf.~\cite{KronheimerMrowka_book_2007}*{Ch.~28.3}, \cite{Manolescu_SWF_spectra_2003}*{p.~909~f.}).
\item
$\HMfrom_*(Y)$ and~$\HMfrom^*(Y)$ vanish in sufficiently high degrees while $\HMto_*(Y)$ and~$\HMto^*(Y)$ vanish in sufficiently low degrees.
\item
$\HMbar_*(Y)$ and~$\HMbar^*(Y)$ are supported in degrees~$\mf q^\mathrm{ev}(Y)=2\Z-2n(Y)$ and are isomorphic to~$\kk$ in these degrees.
\end{enumerate}
The properties~(HM4)--(HM6) rely on the assumption that~$b_1(Y)=0$ while everything else holds more generally.
All of these properties are either proved explicitly in~\cite{KronheimerMrowka_book_2007} or easily derived.
Alternatively, they follow from statements in the subsequent sections (specifically, \cref{T:SWF Borel cohomology}, \cref{T:Localization}, diagram~\eqref{eq:duality diagram HM}, \cref{T:Lidman-Manolescu cohomology}).

\paragraph{The monopole $h$--invariants.}
By the properties (HM5) and~(HM6) above, the map~$p_*$ vanishes in sufficiently high degrees and is an isomorphism in sufficiently low degrees.
In particular, $p_*$ is non-zero and surjective in sufficiently low degrees in~$\mf q^\mathrm{ev}(Y)$.
Similarly, $i^*$ vanishes in low degrees and is an isomorphism in high degrees.
This gives meaning to the following definition which is a mild generalization of~\cite{KronheimerMrowkaOzsvathSzabo_lens_space_surgeries_2007}*{Def.~2.11}:

\begin{definition}\label{D:monopole h-invariants}
Let~$Y$ be a closed \spinc 3--manifold with~$b_1(Y)=0$ and~$\kk$ a commutative ring with unit.
We define the homological and cohomological \emph{monopole $h$--invariants} as
	\begin{align}
	h_w(Y;\kk) &= \tfrac12 \max\Set{q\in\mf q^\mathrm{ev}(Y)}{p_{q-1}\ne0} \\
	h_s(Y;\kk) &= \tfrac12 \max\Set{q\in\mf q^\mathrm{ev}(Y)}{\text{$p_{q-1}$ surjective}} \\
	h^w(Y;\kk) &= \tfrac12 \max\Set{q\in\mf q^\mathrm{ev}(Y)}{i^q\ne0} \\
	h^s(Y;\kk) &= \tfrac12 \max\Set{q\in\mf q^\mathrm{ev}(Y)}{\text{$i^q$ surjective}}
	\end{align}
For a field~$\F$ we write~$h^\F(Y)=h^w(Y;\F)=h^s(Y;\F)$ and~$h^p(Y)=h^{\F_p}(Y)$ and similarly~$h_\F(Y)$ and~$h_p(Y)$ for the homological versions.
\end{definition}

\begin{remark}\label{R:why so many definitions}
The homological and cohomological $h$--invariants are equivalent as functions of~$Y$ (see \cref{T:monopole h-inv hom vs coh} below), which explains the strong homological bias of the Floer literature.
However, we find it more convenient to have both versions available.
Similarly, we find the distinction between the strong and weak flavors helpful, although we will see in \cref{T:coeff dep monopoles} that the invariants~$h_p(Y)$ are sufficient for all practical purposes.
\end{remark}

\begin{lemma}\label{T:monopole h-inv hom vs coh}
Let~$-Y$ be the 3--manifold obtained by reversing the orientation of~$Y$ equipped with the \spinc structure with spinor bundle~$(S,-\rho)$. 
Then
	\begin{equation}\label{eq:h-invariants homology vs cohomology}
	h_{s/w}(Y;\kk) = -h^{s/w}(-Y;\kk).
	\end{equation}
\end{lemma}
\begin{proof}
According to \cite{KronheimerMrowka_book_2007}*{Cor.~22.5.10 and Prop.~28.3.4}, for each~$q\in \mf q(Y)$ there is a diagram with vertical isomorphisms
	\begin{equation}\label{eq:duality diagram HM}\begin{tikzcd}[column sep=small]
	\HMbar_{q-1}(-Y) \rar{i_*} \dar{\cong} &
	\HMto_{q-1}(-Y) \rar{j_*} \dar{\cong} &
	\HMfrom_{q-1}(-Y) \rar{p_*} \dar{\cong} &
	\HMbar_{q-2}(-Y) \dar{\cong} 
	\\
	\HMbar^{-q-1}(Y) \rar{p^*} &
	\HMfrom^{-q}(Y) \rar{j^*} &
	\HMto^{-q}(Y) \rar{i^*} &
	\HMbar^{-q}(Y) 
	\end{tikzcd}\end{equation}
which commutes up to sign. 
This immediately implies~\eqref{eq:h-invariants homology vs cohomology}.
\end{proof}
\begin{remark}\label{R:grading convention}
The diagram above also explains the grading conventions in \cref{D:monopole h-invariants}.
Furthermore, it indicates that~$\mf q(-Y)=-\mf q(Y)$ which follows from~$n(-Y)=-n(Y)$.
The latter is clear from the descriptions of~$n(Y)$ in~\cite{Manolescu_SWF_spectra_2003}*{p.~909~f.}.
\end{remark}

\section{Topological h--invariants}
\label{ch:topological h-invariants}

Recall from \cref{D:SWF spaces intro} that a $\TT$--space of type SWF is a semi--free $\TT$--space~$X$ with~$X^\TT\simeq S^\ell$ for some~$\ell$ which is $\TT$--homotopy equivalent to a finite $\TT$--complex. 
We write~$i_X\colon X^\TT\hra X$ for the fixed point inclusion.
The Borel cohomology of these spaces has a rather simple structure.

\begin{lemma}\label{T:SWF Borel cohomology}
Let~$X$ be a $\TT$--space of type~SWF with~$X^\TT\simeq S^\ell$ and~$\kk$ is commutative ring with unit.
\begin{enumerate}[(i)]
\item
$\BC^*(X;\kk)$ is finitely generated as a $\kk[u]$--module.
\item
$\BC^*(X^\TT;\kk)\cong \kk[u]\, s_X$ with a generator $s_X\in \BC^\ell(X^\TT)\cong\kk$.
\item
$i_X^*\colon \BC^*(X;\kk)\to\BC^*(X^\TT;\kk)$ is an isomorphism in sufficiently high degrees.
\item
$H_\TT^*(X,X^\TT;\kk)$ is torsion over~$\kk[u]$.
\end{enumerate}
\end{lemma}
\begin{proof}
We may assume that~$X$ is a finite $\TT$--complex.
Then~(i) is easily proved by induction over the skeletal filtration and (ii) follows from (B\ref{fact: trivial action}) applied to~$X^\TT$.
Similarly, (iv) follows from (B\ref{fact: free action}) applied to~$Q_X=X/X^\TT$, on which~$\TT$ acts freely away from the base point, with the observation that 
	\begin{equation}
	H_\TT^*(X,X^\TT) \cong \BC^*(Q_X) \cong \tH^*(Q_X/\TT)
	\end{equation}
is non-zero in at most finitely many degrees, since~$Q_X/\TT$ is a finite CW--complex in the ordinary sense.
Lastly, (iii) follows from the long exact sequence in Borel cohomology for the pair~$(X,X ^\TT)$.
\end{proof}

\paragraph{Topological h--invariants.}
As a consequence of~\cref{T:SWF Borel cohomology}, the map~$i_X^*$ vanishes in degrees below~$\ell$ and is non-zero and surjective in degrees~$\ell+2k$ with~$k\ge0$ sufficiently large.
This justifies \cref{D:topological h-invariants intro} which we repeat for convenience.
\begin{definition}\label{D:h invariants topological}
Let~$X$ be a $\TT$--space of type SWF with~$X^\TT\simeq S^\ell$ and~$\kk$ a commutative ring with unit.
The integers
	\begin{align*}
	h^s(X;\Bbbk) &= \min\Set{k\ge0}{\text{$\BC^{\ell+2k}(X;\Bbbk) \xra{i_X^*} \BC^{\ell+2k}(X^\TT;\Bbbk)$ is surjective}}\\
	h^w(X;\Bbbk) &= \min\Set{k\ge0}{\text{$\BC^{\ell+2k}(X;\Bbbk) \xra{i_X^*} \BC^{\ell+2k}(X^\TT;\Bbbk)$ is non-zero}}
	\end{align*}
are called the \emph{strong} and \emph{weak (cohomological) $h$--invariants} of~$X$  with coefficients in~$\kk$.
For coefficients in a field~$\F$ we only have one $h$--invariant
	\begin{equation}
	h^\F(X) = h^w(X;\F) = h^s(X;\F)
	\end{equation}
and for prime fields~$\F_p$ we abbreviate further to~$h^p(X)=h^{\F_p}(X)$.
\end{definition}

We will explain the precise relation to \cref{D:monopole h-invariants} in \cref{ch:relation to SW} after establishing some basic properties of the topological $h$--invariants.

\paragraph{The tower interpretation.}
We first record a convenient interpretation of the $h$-invariants for field coefficients.
Recall that~$\F[u]$ is a principal ideal domain and that every finitely generated~$\F[u]$--module~$M$ is isomorphic to a sum of cyclic modules of the form~$\F[u]/(u^k)$ with~$k\ge0$.
The summands with~$k>0$ constitute the torsion submodule.
If $M$ is further $\Z$--graded over~$\F$ such that~$u$ has degree~$\pm2$, then each cyclic summand can be visualized as a `tower' starting in certain degree.
It is common folklore in the 3--manifold community that~$\HMfrom_*(Y;\F)$ has exactly one infinite tower which is generated in a certain degree that is measured by~$h_\F(Y)$.
The same holds for the topological $h$--invariants.

\begin{lemma}\label{T:SWF Borel cohomology fields}
Let~$X$ be a $\TT$--space of type SWF with~$X^\TT\simeq S^\ell$.
Further, let~$\F$ a field and write~$h=h^\F(X)$.
Then for any $0\ne s_X\in \BC^\ell(X^\TT;\F)$ we have
	\begin{equation}
	i_X^*\BC^*(X;\F) = \F[u]\,u^hs_X
	\end{equation}
and there is a splitting 
	\begin{equation}
	\BC^*(X;\F) = \F[u]\,\tilde s_X \oplus T\BC^*(X;\F)
	\end{equation}
where $\tilde s_X\in\BC^{\ell+2h}$ satisfies~$i_X^*\tilde s_X = u^hs_X$ and $T\BC^*(X;\F)$ is the $\F[u]$--torsion submodule.
In particular, $\BC^*(X;\F)$ has rank one over~$\F[u]$ and the free part is generated in degree~$h^\F(X)$.
\end{lemma}
\begin{proof}
This immediately follows from \cref{T:SWF Borel cohomology} and the fact that~$\F[u]$ is a principal ideal domain.
\end{proof}
\begin{remark}\label{R:towers only for fields}
If~$\kk$ is not a field, then the structure of $\kk[u]$--modules can be more complicated and possible tower interpretations of~$h^w(X;\kk)$ and~$h^s(X;\kk)$ are less straightforward and arguably less illuminating.
The exception is the case when~$i_X^*\BC^*(X,\kk)$ is free as a~$\kk[u]$--module and can thus be embedded into~$\BC^*(X,\kk)$ as a summand.
The latter is an infinite tower of the form~$\kk[u]$ generated in the degree $\ell+2h^s(X;\kk)$.
We also note that $h^s(X;\kk)=h^w(X;\kk)$ in this case.
\end{remark}

\paragraph{A comparison result.}
One of our main objectives is to investigate the dependence of the invariants~$h^{s/w}(X;\kk)$ on the choice of~$\kk$ and the flavor.
By definition, we have an inequality
	\begin{equation}
	h^w(X;\kk)\le h^s(X;\kk)
	\end{equation}
which can be strict, in general, as we will see in \cref{eg:example for h-difference} below.
The next result highlights the role of the invariants~$h^p(X)$ derived from the prime fields.
Recall that every principal ideal domain is an integral domain and thus has a field of fractions.

\begin{proposition}\label{T:basic inequalities h-invariant}\mbox{}
Let~$X$ be a $\TT$--space of type~SWF and~$\kk$ a principal ideal domain.
\begin{enumerate}[(i)]
\item
If $\F$ is the field of fractions of~$\kk$, then $h^\F(X)=h^w(X;\kk)$.
\item
If~$\F$ is an arbitrary field of characteristic~$p$, then $h^\F(X)=h^p(X)$.
\item
We have $h^s(X;\kk)\le h^s(X;\Z)$.
\item
For coefficients in~$\Z$ we have
	\begin{equation}\label{eq:h-invariants integers}
	h^w(X;\Z)=h^0(X) \eqand h^s(X;\Z) = \max_p h^p(X).
	\end{equation}
Moreover, the equality $h^p(X)=h^s(X;\Z)$ either holds for all~$p$ or for at most finitely many~$p$.
\end{enumerate}
\end{proposition}
\begin{proof}
For every $\kk$--module~$M$ a version of the ordinary universal coefficient theorem applied to~$\ETwt X$ gives short exact sequences
	\begin{equation}\label{eq:UCT for BC}
	0\to \BC^n(X;\kk)\otimes_\kk M \to \BC^n(X;M) \to \Tor_\kk\big(\BC^{n+1}(X;\kk),M\big) \to0
	\end{equation}
for all~$n\ge0$ (see~\cite{tomDieck_algebraic_topology_2008}*{(11.9.6)}, for example).
The statements (i)--(iv) can be proved using different instances of~\eqref{eq:UCT for BC}.
Since~$X$ is fixed while the coefficients vary, we write~$i_\kk^*\colon\BC^*(X;\kk)\to\BC^*(X^\TT;\kk)$ for the maps induced by the fixed point inclusion.

\medskip
(i) Since~$\F$ is flat as a $\kk$--module (e.g.~\cite{Rotman_homological_algebra_2nd_2009}*{Cor.~5.35}), the $\Tor$~term in~\eqref{eq:UCT for BC} vanishes and we have a commutative diagram
	\begin{equation*}\begin{tikzcd}
	\BC^{\ell+2n}(X;\kk)\otimes_\kk\F \rar{i^*_\kk\otimes1} \dar{\cong}& 
		\BC^{\ell+2n}(X^\TT;\kk)\otimes_\kk\F \dar{\cong} \\ 
	\BC^{\ell+2n}(X;\F) \rar{i^*_\F} & 
		\BC^{\ell+2n}(X^\TT;\F)
	\end{tikzcd}\end{equation*}
in which the vertical maps are isomorphisms.
This shows that~$i^*_\F\ne 0$ if and only if~$i^*_\kk\otimes1\ne0$.
Moreover, it is elementary that~$i^*_\kk\otimes1\ne0$ if and only if~$i^*_\kk\ne 0$.
If follows that~$h^w(X;\kk)=h^w(X;\F)$.
\medskip

(ii) Consider~$\F$ as a vector space over~$\F_p$.
Again, the $\Tor$~term in~\eqref{eq:UCT for BC} vanishes and we can argue as above using the diagram
	\begin{equation*}\begin{tikzcd}
	\BC^{\ell+2n}(X;\F_p)\otimes_\kk \F \rar{i^*_{\F_p}\otimes1} \dar{\cong}& 
		\BC^{\ell+2n}(X^\TT;\F_p)\otimes_\kk\F \dar{\cong} \\ 
	\BC^{\ell+2n}(X;\F) \rar{i^*_\F} & 
		\BC^{\ell+2n}(X^\TT;\F).
	\end{tikzcd}\end{equation*}
\medskip

(iii) If we consider~$\kk$ as a $\Z$--module, we get a similar commutative diagram
	\begin{equation*}\begin{tikzcd}
	\BC^{\ell+2n}(X;\Z)\otimes_\Z \kk \rar{i^*_{\Z}\otimes1} \dar& 
		\BC^{\ell+2n}(X^\TT;\Z)\otimes_\Z\kk \dar{\cong} \\ 
	\BC^{\ell+2n}(X;\kk) \rar{i^*_\kk} & 
		\BC^{\ell+2n}(X^\TT;\kk),
	\end{tikzcd}\end{equation*}
but the left vertical map is no longer guaranteed to be an isomorphism.
Nevertheless, we still have an isomorphism on the right and the surjectivity of~$i^*_\Z$ implies that of~$i^*_\Z\otimes1$, since the functor~$\_\otimes_\Z\kk$ is right exact.
By commutativity, the map~$i^*_\kk$ is also surjective and therefore we must have~$h^s(X;\kk)\le h^s(X;\Z)$.
\medskip

(iv)
The equality $h^w(X\;\Z)=h^0(X)$ is a special case of~(i)
and we know from~(iii) that $h^p(X)\le h^s(X;\Z)$ for all~$p$.
We have to exhibit at least one~$p$ for which~$h^p(X)=h^s(X;\Z)$.
For brevity, we write~$\nu=\ell+2h^s(X;\Z)$ and note that~$h^p(X)<h^s(X;\Z)$ if and only if~$i^*_{\F_p}$ is non-zero in degree~$\nu-2$.

Note that for arbitrary~$\kk$, we have an exact sequence
	\begin{equation}
	\BC^{\ell+2k}(X;\kk) \xra{i^{\ell+2k}_\kk} 
	\BC^{\ell+2k}(X^\TT;\kk) \xra{\delta^{\ell+2k}_\kk}
	\BC^{\ell+2k+1}(X,X^\TT;\kk) 
	\end{equation}
showing that $i^*_\kk$ is surjective if and only if~$\delta_\kk$ vanishes.
This gives alternative descriptions
	\begin{equation}
	h^s(X;\kk) 
	= \min\Set{k\ge0}{\delta^{\ell+2k}_\kk=0} 
	= \max\Set{k\ge0}{\delta^{\ell+2k-2}_\kk\ne0}.
	\end{equation}
Consider the following commutative diagram:
	\begin{equation}\label{eq:Z vs F_p}\begin{tikzcd}
	\BC^{\nu-2}(X;\Z)\otimes_\Z\F_p \rar{i^{\nu-2}_\Z\otimes 1_p} \ar[d,hook] &
		\BC^{\nu-2}(X^\TT;\Z)\otimes_\Z\F_p \rar{\delta^{\nu-2}_\Z\otimes 1_p} \dar{\cong} &
			\BC^{\nu-1}(X,X^\TT;\Z)\otimes_\Z\F_p \ar[d,hook] \\
	\BC^{\nu-2}(X;\F_p) \rar{i^{\nu-2}_{\F_p}} &
		\BC^{\nu-2}(X^\TT;\F_p) \rar{\delta^{\nu-2}_{\F_p}} &
			\BC^{\nu+1}(X,X^\TT;\F_p)^{\nu-2}
	\end{tikzcd}\end{equation}
The vertical maps are injective or isomorphisms as indicated, as can be seen from~\eqref{eq:UCT for BC}.
The second row in~\eqref{eq:Z vs F_p} is always exact while the first may fail to be exact for~$p>0$.
We can conclude that
	\begin{equation}
	i^{\nu-2}_{\F_p}\ne0 \quad\Leftrightarrow\quad 
	\delta^{\nu-2}_{\F_p}=0 \quad\Leftrightarrow\quad 
	\delta^{\nu-2}_\Z\otimes1_p=0
	\end{equation}
where the first equivalence holds since~$\F_p$ is a field and the second uses the commutativity of the right square in the diagram.
We thus have to understand the map~$\delta_\Z\otimes1_p$.
For that purpose, we fix a generator~$\xi \in \BC^{\nu-2}(X^\TT;\Z)\cong\Z$ and observe that~$\delta_\Z(\xi)$ generates a cyclic subgroup of~$\BC^{\nu-1}(X,X^\TT;\Z)$ isomorphic to~$Z/k\Z$ for some~$k\ge0$.
By assumption, $i^{\nu-2}_\Z$ is not surjective so that have~$k\ne1$.
We further have
	\begin{equation}
	\im(\delta_\Z\otimes1_p)\cong
	\im(\delta_\Z)\otimes\F_p\cong 
	\Z/k\Z\otimes_\Z \F_p \cong
	\begin{cases}
	\F_p &\text{if $k\equiv0\mod p$} \\
	0 & \text{else}.
	\end{cases}
	\end{equation}
and arrive at the following conclusion:
	\begin{equation}
	h^p(X)=h^s(X;\Z) \quad\Leftrightarrow\quad 
	\delta^{\nu-2}_\Z\otimes1_p\ne0 \quad\Leftrightarrow\quad 
	k\equiv0\mod p.
	\end{equation}
The last condition clearly holds for at least one~$p$.
Indeed, if~$k=0$, then it holds for all~$p$, and if~$k>0$, it holds precisely for the finitely many prime factors of~$k$.
This completes the proof of~(iv).
\end{proof}

\section{Some examples}
\label{ch:examples}

We will now construct a family of $\TT$--spaces~$X_{a,b}$ of type~SWF indexed by integers~$a,b\in\Z$ such that such that for suitable values of~$a,b$ the invariants~$h^{s/w}(X_{a,b};\kk)$ differ depending of flavor and the choice of coefficients.
As before, $\kk$ be any commutative ring with unit.
We will frequently refer use the properties (B\ref{fact:module structure})--(B\ref{fact: Borel vs colimits}) of Borel cohomology listed 
on page~\pageref{fact:module structure}.

\paragraph{The starting point.} 
The simplest examples of $\TT$--spaces of type~SWF are the spheres~$S^{\ell,h}$ associated to the $\TT$--representations $\R^\ell\oplus\C^h$.
Clearly, $(S^{\ell,h})^\TT\cong S^\ell$ and the $\TT$--action is free away from the fixed points.
Using~(B\ref{fact: suspension isomorphisms}) we find
	\begin{equation}
	\BC^*(S^{\ell,h};\kk)\cong\kk[u]\, e_{\ell,h}
	\AND
	\BC^*(S^\ell;\kk)\cong\kk[u]\, s_\ell
	\end{equation}
with generators~$e_{\ell,h}\in \BC^{\ell+2h}(S^{\ell,h};\kk)$ and $s_\ell\in \BC^\ell(S^\ell;\kk)$ and we may assume that $i^*e_{\ell,h}=u^hs_\ell$ where~$i$ is the fixed points inclusion.
For the latter, note that $e_{\ell,h}$ and~$s_\ell$ are Thom classes for the vector bundles over~$\BT$ associated to the representations~$\R^\ell\oplus\C^h$ and~$\R^\ell$.
From this we can deduce
	\begin{equation}
	h^w(S^{\ell,h};\kk)=
	h^s(S^{\ell,h};\kk)=h
	\end{equation}
In particular, all possible $h$--invariants agree in this case.

\paragraph{A trivial cell attachment.}
Starting with~$S^{\ell,h}$, we write~$n=\ell+2h$ and attach a $\TT$--free $n$--cell to obtain the space
	\begin{equation}
	A = S^{\ell,h} \vee (\TT_+\wedge S^{n}).
	\end{equation}
This is again a $\TT$--space of type~SWF with $A^\TT\cong S^\ell$.
We can compute the Borel cohomology of~$A$ and~$A^\TT$ using the properties~(B\ref{fact: suspension isomorphisms}), (B\ref{fact: free action}) and~(B\ref{fact: Borel vs colimits}):
as $\kk[u]$--modules they are given by
	\begin{equation}
	\BC^*(A) = \kk[u]\, e_A \oplus \kk\, f_A \AND 
	\BC^*(A^\TT) = \kk[u]\, s_A
	\end{equation}
with generators~$e_A,f_A\in \BC^{n}(A)$, $s_A\in \BC^\ell(A^\TT)$ and~$u$ acting trivially on~$\kk$.
The inclusion $i_A\colon A^\TT\hra A$ maps the generators as
	\begin{equation}
	i_A^*(f_A)=0 \AND 
	i_A^*(e_A)=u^hs_A
	\end{equation}
which follows from the computation for~$S^{\ell,h}$ and the fact that~$uf_A=0$.
The $h$--invariants of~$A$ are thus unaffected by the cell attachment and we find
	\begin{equation}
	h^w(A;\kk) = 
	h^s(A;\kk) = h
	\end{equation}
for all choices of~$\kk$.

\paragraph{A non-trivial cell attachment.}
The construction continues by attaching a $\TT$--free $(n+1)$--cell to~$A = S^{\ell,h} \vee (\TT_+\wedge S^{n})$.
The relevant attaching maps are parameterized by the based $\TT$--homotopy set~$[(\TT\times S^n)_+,A]^\TT$.
For $n\ge2$ we find:

\begin{lemma}\label{T:attaching maps}
Let $A=S^{\ell,h}\vee(\TT_+\wedge S^n)$ with $n=\ell+2h\ge2$ and fix a generator~$e_n\in\BC^n((\TT\times S^n)_+)\cong \tilde{H}^n(S^n)$.
Then the following map is a bijection:
	\begin{equation}
	[(\TT\times S^n)_+,A]^\TT \to\Z\oplus\Z,\quad
	\varphi\mapsto \big(\langle e_V,\varphi_*e_n\rangle, \langle f_V,\varphi_*e_n \rangle \big).
	\end{equation}
\end{lemma}

We postpone the proof to finish the construction of our examples.
For given~$a,b\in\Z$ let~$X_{a,b}$ be the space obtained by attaching a free $\TT$--cell along the a based $\TT$--map~$\varphi_{a,b}\colon (\TT\times S^{n})_+\to A$ with 
	\begin{equation}
	\langle\varphi_{a,b}^*e_A,e_n\rangle = a \AND 
	\langle\varphi_{a,b}^*f_A,e_n\rangle = b.
	\end{equation}
The $h$--invariants of~$X_{a,b}$ with field coefficients are given as follows:
\begin{lemma}\label{T:h-invariants of X_ab}
For~$p$ a prime or zero we have
	\begin{equation}\label{eq:h-invariants of X_ab}
	h_p(X_{a,b}) = 
	\begin{cases}
	h+1, & \text{$b\in p\Z$ and $a\notin p\Z$} \\
	h, & \text{else}
	\end{cases}
	\end{equation}
\end{lemma}

\begin{proof}
For brevity, we write~$X=X_{a,b}$ and~$\varphi=\varphi_{a,b}$.
Let~$\iota\colon A\hra X$ be the inclusion and~$k\ge0$.
For arbitrary coefficients, we have a commutative square
	\begin{equation*}\begin{tikzcd}
	\BC^{\ell+2k}(X) \ar[d,"i_X^*"] \ar[r,"\iota^{\ell+2k}"]&
		\BC^{\ell+2k}(A) \ar[d,"i_A^*"] \\
	\BC^{\ell+2k}(X^\TT) \ar[r,"(\iota^\TT)^{\ell+2k}","\cong"']&
		\BC^{\ell+2k}(A^\TT) 
	\end{tikzcd}\end{equation*}
The first observation is that the handle attachment can only effect the Borel cohomology in degrees~$n$ and~$n+1$.
Indeed, by construction we have a cofiber sequence $A\xra\iota X\to \TT_+\wedge S^{n+1}$ and~$\BH(\TT_+\wedge S^{n+1})$ is concentrated in degree~$n+1$.
In particular, $\iota^{\ell+2k}$ is an isomorphism for~$k\ne h$, leaving~$h$ and~$h+1$ as the only possible values of~$h_p(X)$.
To decide which of these values is attained, we embed the square into a larger diagram.
This requires a bit of preparation.

We continue to work with arbitrary~$\kk$.
Consider 
the quotients $Q_A=A/A^\TT$ and~$Q_X=X/X^\TT$ and 
the cofiber sequences
	\begin{equation}
	A^\TT \lrao{i_A} A \lrao{q_A} Q_A \AND 
	X^\TT \lrao{i_X} X \lrao{q_X} Q_X.
	\end{equation}
Since~$A$ and~$X$ have semi-free $\TT$--actions, the actions on~$Q_A$ and~$Q_X$ are free away from the base point.
We also note that~$Q_X$ is obtained from~$Q_A$ by attaching a $\TT$--free cell along~$\bar\varphi=q_A\circ\varphi$.
The cofiber sequence~\eqref{eq:cofiber seq V} gives a $\TT$--homotopy equivalence~$S^{\ell,d}/S^\ell\simeq (S(\C^h)_+\wedge S^1)\wedge S^{\ell}$ and thus
	\begin{equation}
	Q_A \homequ (S(\C^h)_+\wedge S^{\ell+1}) \vee (\TT_+\wedge S^n).
	\end{equation}
Passing to $\TT$--orbits we get a non-equivariant homotopy equivalence
	\begin{equation}
	Q_A/\TT \homequ (\CP^{h-1}_+\wedge S^{\ell+1}) \vee S^n.
	\end{equation}
By (B\ref{fact: free action}) and (B\ref{fact: free action}) the Borel cohomology given 
	\begin{equation}
	\BC^*(Q_A) 
	\cong \tH^*(Q_A/\TT)
	\cong \kk[u]/u^h\, d_{A} \oplus \kk\, \bar f_A
	\end{equation}
with $\kk[u]$--module generators $d_A\in \BC^{\ell+1}(Q_A)$ and $\bar f_A\in\BC^n(Q_A)$ satisfying
	\begin{equation}
	f_A = q_A^*\bar f_A \AND
	d_{A} = \delta_A s_A 
	\end{equation}
where $\delta_A\colon \BC^*(A^\TT)\to\BC^{*+1}(Q_A)$ is the connecting homomorphism.
Consider the following diagram whose rows and columns are exact:
	\begin{equation*}\begin{tikzcd}
	& 0 \dar & 0\dar & \\
	0 \rar &
		\BC^n(Q_X) \ar[r,"\bar \iota^*"] \ar[d,"q_X^*"] &
			\BC^n(Q_A) \ar[r,"\bar\varphi^*"] \ar[d,"q_A^*"] &
				\BC^n((\TT\times S^n)_+) \ar[d,equal]\\
	0 \rar &
		\BC^n(X) \ar[r,"\iota^*"] \ar[d,"i_X^*"] &
			\BC^n(A) \ar[r,"\varphi^*"] \ar[d,"i_A^*"] &
				\BC^n((\TT\times S^n)_+) \\
	0 \rar &
		\BC^n(X^\TT) \ar[r,"\cong"] &
			\BC^n(A^\TT) \rar & 0
				\\
	&& 0 \ar[from=u] &
	\end{tikzcd}\end{equation*}
Let~$e^n\in \BC^n((\TT\times S^n)_+)$ be the generator with $\langle e^n,e_n \rangle=1$.
Since~$i_A^*e_A=s_A$ and~$i_A^*f_A=0$, the preimage of~$s_A$ under~$i_A^*$ consists of the classes~$e_A+\kk f_A$ and by assumption we have
	\begin{equation}
	\varphi^*e_A = a_\kk\,e^n \AND 
	\varphi^*f_A = \bar \varphi^* \bar f_A = b_\kk\,e^n
	\end{equation}
where~$a_\kk, b_\kk$ are the images under the canonical map~$\Z\to\kk$.
Specializing to~$\kk=\F$ a field, we can deduce that~$i_X^*$ is surjective unless~$b_\F=0$ and~$a_\F\ne0$.
For the prime fields~$\F_p$ this is equivalent to the condition in~\eqref{eq:h-invariants of X_ab}.
\end{proof}

From here on, it is easy to find~$a,b\in\Z$ such that~$h_p(X_{a,b})$ depends non-trivially on the choice of~$p$.
For concreteness, we record:

\begin{example}\label{eg:example for h-difference}
Let~$a,b\in\N$ be distinct primes. Then \cref{T:h-invariants of X_ab} gives
	\begin{equation}
	h_p(X_{a,b}) = 
	\begin{cases}
	h+1, & p=b\\
	h, & \text{else.}
	\end{cases}
	\end{equation}
Combining this with \cref{T:basic inequalities h-invariant}(iv) gives
	\begin{equation}
	h^w(X_{a,b};\Z)=h^0(X_{a,b})=h<h+1=h^b(X_{a,b})=h^s(X_{a,b};\Z).
	\end{equation}
\end{example}

It remains to prove \cref{T:attaching maps}.

\paragraph{Proof of \cref{T:attaching maps}.}

We need a technical lemma first.
Recall that the Borel construction~$\ET\times_\TT X$ is a fiber bundle over~$\BT$ with fiber~$X$ and that a choice of~$e\in\ET$ gives rise to an embedding~$j\colon X\hra \ET\times_\TT X$.
Composing with the quotient map to~$\ET_+\wedge_\TT X$ gives a based embedding $\bar j\colon X\to \ET_+\wedge_\TT X$ which induces comparison maps for Borel and ordinary (co-)homology:
	\begin{equation}\label{eq:comparison maps}
	\bar j^*\colon \BC^*(X) \to \tH^*(X) \AND 
	\bar j_*\colon \tH_*(X) \to \BH_*(X)
	\end{equation}

\begin{lemma}\label{T:fiber inclusion}
Let~$X$ be either~$S^V$ or $\TT_+\wedge Z$ where
\begin{itemize}
\item
$V$ is an $n$--dimensional $\TT$--representation and
\item 
$Z$ is a based $(n-1)$--connected space~$Z$ with trivial $\TT$--action.
\end{itemize}
Then the comparison map $\bar j^*\colon \BC^n(X)\to \tH^n(X)$ is an isomorphism.
\end{lemma}

\begin{proof}
Both $S^V$ and~$\TT_+\wedge Z$ are non-equivariantly $(n-1)$--connected.
The Serre spectral sequence from~(B\ref{fact: Serre spectral sequence}) for~$X$ thus gives the exactness of the middle row in the following commutative diagram:
	\begin{equation}\label{eq:fiber inclusion diagram}
	\begin{tikzcd}[column sep=2em]
	&& H^n(\Borr X) \ar[d,"q^*"] \ar[dr,"\bar j^*"] && \\
	0\rar &
		H^n(\BT) \rar{p^*} &
			H^n(\Bor X) \rar{j^*} \dar{s^*}&
				H^n(X) \rar{d_{n+1}} &
					H^{n+1}(\BT) \\
	&& H^n(\BT)  \ar[ul,equal] && 
	\end{tikzcd}\end{equation}
The map~$d_{n+1}$ is the edge differential
	$H^n(X) \cong E_{n+1}^{0,n} \ra E_{n+1}^{n+1,0}\cong H^{n+1}(\BT)$.

If~$X=S^V$, this is the transgression of the oriented sphere bundle $\ET\times_\TT S^V$ over~$\BT$ whose image is generated by the Euler class (cf.~\cite{Rudyak_Thom_spectra_etc_1998}*{Prop.~V.1.26}).
Since the bundle has a section, its Euler class vanishes, and so does~$d_{n+1}$.
Thus~$j^*$ is surjective and the splittings of the horizontal and vertical sequences given by~$s^*$ and~$p_*$ show that~$\bar j^*$ is an isomorphism.

For~$X=\TT_+\wedge Z$ we consider the composition
	\begin{equation*}\begin{tikzcd}[row sep=small]
	\TT_+\wedge Z \ar[r,"\bar j"] &
		\Borr{(\TT_+\wedge Z)} \ar[r,"\simeq"] &
			(\TT_+\wedge Z)/\TT \ar[r,"\cong"] &
				Z \\
	g\wedge z \ar[r,maps to] &
		{[e\wedge g\wedge z]} \ar[r,maps to] &
			{[g\wedge z]} \ar[r,maps to] &
				g\inv z=z
	\end{tikzcd}\end{equation*}
Since~$\TT$ acts freely on~$\TT_+\wedge Z$ away from the base point, the middle map is a homotopy equivalence by~(B\ref{fact: free action}) and the map on the right is clearly a homeomorphism.
The resulting map~$\TT_+\wedge Z\to Z$ has a section given by $z\mapsto 1\wedge z$ so that~$\bar j^*$ is split surjective.
Lastly, $q^*$ is split injective by~(B\ref{fact: reduced vs unreduced}), and~\eqref{eq:fiber inclusion diagram} shows that $j^*$ is injective on the image of~$q^*$.
It follows that~$\bar j^*$ is an isomorphism.
The statement for~$\bar j_*$ can be proved along the same lines.
\end{proof}

\begin{proof}[Proof of \cref{T:attaching maps}]
We first argue that the map $\varphi\mapsto \varphi_*e_n$ factors as a series of bijections
	\begin{equation}\label{eq:bijections for attachment}
	[(\TT\times S^n)_+,A]^\TT
	\cong [S^n_+,A]
	\cong \pi_n(A)
	\cong \tH_n(A;\Z)
	\cong \BH_n(A;\Z).
	\end{equation}
From left to right, the maps are
	a standard adjunction (e.g.~\cite{tomDieck_transformation_groups_1987}*{I.(4.9)}),
	the forgetful map from~$\pi_n(A)$ to unbased homotopy classes,
	the Hurewicz map, and
	the comparison map~$\bar j_*\colon\tH_n(A)\to \BH_n(A)$ from~\eqref{eq:comparison maps}.
The adjunction is always bijective. 
So are the forgetful and Hurewicz maps, since~$A$ is non-equivariantly $(n-1)$--connected with~$n\ge 2$.
Lastly, $\bar j_*$~is an isomorphism by \cref{T:fiber inclusion} and the additivity of Borel homology are additive under wedge sums.
Unraveling the definitions shows that the composition in~\eqref{eq:bijections for attachment} is indeed the map~$\varphi\mapsto\varphi_*e_n$.

A similar computation as that of~$\BC^*(A;\Z)$ shows that~$\BH_*(A;\Z)$ is free over~$\Z$ in each degree and~$\BH_n(A)\cong\Z\oplus\Z$.
The lemma now follows from the standard universal coefficient theorem for~$\ET_+\wedge_\TT A$ and the fact that~$e_A$ and $f_A$ generate~$\BC^n(A;\Z)$.
\end{proof}

\section{Duality, additivity, and monotonicity}
\label{ch:additivity etc}

In this section we derive several properties of the topological $h$--invariants from which are analogues of the properties of the monopole $h$--invariants listed in \cref{T:Froyshov intro,T:additivity etc intro}.

\subsection{A description in terms of Tate cohomology}

\label{ch:Tate etc}
While the elementary definition of the topological $h$--invariants in terms of fixed points has its merits, it is often more convenient to use a different description in terms of Borel and Tate (co-)homology as defined by Greenlees and May in~\cite{GreenleesMay_equivariant_Tate_1995}.
This requires some familiarity with equivariant stable homotopy theory and, through~\cite{GreenleesMay_equivariant_Tate_1995}, implicitly makes use of the genuine stable $\TT$--equivariant homotopy category of spectra developed in~\cite{LewisMaySteinberger_equivariant_stable_homotopy_1986}.
Throughout, $X$ will be a finite $\TT$--complex and we use the same letter for its suspension spectrum.

\paragraph{Borel, Tate, and coBorel theories.}
Let~$\kk$ be a principal ideal domain and~$\Hk$ its non-equivariant Eilenberg--Mac~Lane spectrum, considered as a \mbox{$\TT$--spectrum} with trivial action.
Following~\cite{GreenleesMay_equivariant_Tate_1995}*{p.~2~ff.}, we consider the associated $\TT$--spectra
	\begin{equation}\label{eq:FCT definition}
	F = F(E\TT_+,\Hk) \wedge E\TT_+,\quad
	C = F(E\TT_+,\Hk),\quad
	T = F(E\TT_+,\Hk)\wedge \widetilde{E\TT}
	\end{equation}
which are related by a cofiber sequence of $\TT$--spectra
	\begin{equation}\label{eq:FCT cofiber sequence}
	F\lrao{\mu} C \lrao{\lambda}  T \lrao{\delta} F\wedge S^1\lra \cdots 
	\end{equation}
Recall that every $\TT$--spectrum~$E$ represents $\TT$--equivariant homology and cohomology theories which are defined as
	\begin{equation}
	E_nX = [S^n,E\wedge X]^\TT \AND 
	E^nX = [X,E\wedge S^n]^\TT,\quad n\in\Z
	\end{equation}
where~$[\,,\,]^\TT$ indicates morphism sets in the stable homotopy category.
According to \cite{GreenleesMay_equivariant_Tate_1995}*{Prop.~2.1}, the $\TT$--spectra~$C$ and~$F$ represent $\TT$--equivariant Borel cohomology and homology, respectively, in the sense that there are natural isomorphisms
	\begin{equation}\label{eq:representing Borel}
	C^nX     \cong \BC^n(X;\kk) \AND 
	F_{n+1}X \cong \BH_n(X;\kk).
	\end{equation}
The degree shift on the right hand side occurs, because the isomorphism involves a suspension by adjoint representation~$Ad(\TT)\cong\R$.
The theories represented by~$T$ are called \emph{Tate (co-)homology} while $C_*X$ and~$F^*X$ are called \emph{coBorel homology} and~\emph{cohomology}, respectively.
Combining the notation from~\cites{Manolescu_triangulaion_JAMS_2016,LidmanManolescu_equivalance_2018} and the grading conventions from~\cite{GreenleesMay_equivariant_Tate_1995}, we define
	\begin{equation}\begin{split}
	T^nX = \tBC^n(X;\kk), \hspace{15mm} & T_nX = \tBH_{n-1}(X;\kk), \\
	F^nX = \cBC^n(X;\kk), \hspace{15mm} & C_nX = \cBH_{n-1}(X;\kk).
	\end{split}\end{equation}

\begin{remark}\label{R:ring spectra}
All these theories take values in $\kk[u]$--modules.
Indeed, since~$\Hk$ is a commutative ring spectrum, $C$ and~$T$ are ring $\TT$--spectra and~$\lambda$ is a map of ring $\TT$--spectra (see \cite{GreenleesMay_equivariant_Tate_1995}*{Prop.~3.5}).
Along the same lines, one can show that~$F$ is a $C$--module $\TT$--spectrum and that~\eqref{eq:FCT cofiber sequence} is a sequence of $C$--module $\TT$--spectra.
In particular, the represented (co-)homology theories are naturally modules over~$C^*S^0\cong H^*(\BT;\kk)\cong\kk[u]$.
\end{remark}

\paragraph{Localization.}

Alternatively, the Tate theories~$T^*X$ and~$T_*X$ can be obtained by localizing the $\kk[u]$--modules~$C^*X$ and~$C_*X$ away from~$u$ (cf.~\cite{GreenleesMay_equivariant_Tate_1995}*{Ch.~16}).
This is most transparent for cohomology and we the following statement.

\begin{proposition}[Localization theorem]\label{T:Localization}
Let~$X$ be a finite $\TT$--complex.
\begin{enumerate}[(i)]
\item
There are natural isomorphisms
	\begin{equation}
	\tBC^*(X;\kk) \cong u\inv\BC^*(X;\kk) \AND 
	\tBH_*(X;\kk) \cong u\inv\cBH_*(X;\kk)
	\end{equation}
where~$u\inv$ indicates the localization of $\kk[u]$--modules at the multiplicative subset~$\{1,u,u^2,\dots\}$.
\item
If~$X$ is semi-free, then the fixed point inclusion $i_X\colon X^\TT\hra X$ induces an isomorphism~$\tBC^*(X)\xra{\cong} \tBC^*(X^\TT)$ and thus
	\begin{equation}
	\tBC^*(X) \cong \tH^*(X^\TT;\kk)\otimes_\kk \kk[u,u\inv].
	\end{equation}
The same holds for~$\tBH_*(X)$.
\end{enumerate}
\end{proposition}
\begin{proof}
The isomorphisms in~(i) are special cases of~\cite{GreenleesMay_equivariant_Tate_1995}*{Cor.~16.7 and~16.7}, since~$\TT$ acts freely on the unit sphere of~$\C$ and~$u\in H^2(\BT)$ is the Euler class of the vector bundle~$\ET\times_\TT\C\to\BT$.
To establish~(ii), we use the cofiber sequence~$X^\TT\xra{i_X}X\xra{q_X} X/X^\TT$ and its long exact sequence in Tate cohomology.
Since~$X$ is semi--free, $X/X^\TT$ is free in the based sense so that~$\tBC^*(X/X^\TT)$ by~\cite{GreenleesMay_equivariant_Tate_1995}*{Proposition~2.4}.
\end{proof}

As a consequence, we obtain an alternative description of the topological $h$--invariants.

\begin{corollary}\label{T:h via Tate}
Let~$X$ be a $\TT$--space of type SWF with~$X^\TT\simeq S^\ell$.
Then
	\begin{align}
	h^w(X;\Bbbk) 
	&= \min\Set{k\in\Z}{\text{$\tilde{H}^{\ell+2k}_\TT(X)\xra{\lambda^*} t\tilde{H}^{\ell+2k}_\TT(X)$ is non-zero}} \label{eq:h_w via Tate}
	\\[1em]
	h^s(X;\Bbbk) 
	&= \min\Set{k\in\Z}{\text{$\tilde{H}^{\ell+2k}_\TT(X)\xra{\lambda^*} t\tilde{H}^{\ell+2k}_\TT(X)$ is surjective}} \label{eq:h_s via Tate}
	\end{align}
\end{corollary}
\begin{proof}
\cref{T:Localization} gives a commutative diagram with isomorphisms as indicated:
\begin{equation}\begin{tikzcd}
	\tilde{H}^*_\TT(X) \ar[d,"i_X^*"] \ar[r,"\lambda^*"] &
		t\tilde{H}^*_\TT(X) \ar[d,"i_X^*","\cong"'] \ar[r,"\cong"] &
			u\inv\tilde{H}^*_\TT(X) \ar[d,"i_X^*","\cong"'] \\
	\tilde{H}^*_\TT(X^\TT) \ar[r,"\lambda^*"] &
		t\tilde{H}^*_\TT(X^\TT) \ar[r,"\cong"] &
			u\inv\tilde{H}^*_\TT(X^\TT) 
	\end{tikzcd}\end{equation}
First observe that $\lambda^*\colon \BC^{\ell+2k}(X^\TT)\to t\BC^{\ell+2k}(X^\TT)$ is necessarily trivial for~$k<0$ and an isomorphism for~$k\ge0$.
For $\xi\in \tilde{H}^{\ell+2k}_\TT(X)$ it follows that~$\lambda^*\xi\ne 0$ if and only if~$i_X^*\xi\ne0$, in which case we must have~$k\ge0$.
This shows that the minimum in~\eqref{eq:h_w via Tate} is attained and, in fact, agrees with~$h^w(X;\kk)$.
Similarly, $\lambda^*\xi$ is a generator if and only if~$i_X^*\xi\ne0$ is a generator which gives~\eqref{eq:h_s via Tate}.
\end{proof}

\subsection{Spanier--Whitehead duality}

We discuss two types of duality.
The first is a geometric nature while the second follows formally from the multiplicative structures of the $\TT$--spectra~$F$, $C$, and~$T$.

\begin{definition}[cf.~\cite{May_equivariant_homotopy_1996}*{\S~XVI.8}]\label{D:V-duality}
Based $\TT$--spaces~$X$ and~$X^*$ are \emph{(Spanier--Whitehead) $V$--dual} \wrt a $\TT$--representation~$V$ if there are based $\TT$--maps
	\begin{equation}
	\epsilon\colon X^*\wedge X \to S^V \AND
	\eta\colon S^V\to X\wedge X^*
	\end{equation}
such that the compositions
	\begin{equation}\label{eq:duality maps}\begin{split}
	S^V\wedge X \xra{\eta\wedge1_X}
	X\wedge X^*\wedge &X \xra{1_X\wedge\epsilon}
	X\wedge S^V \\
	X^*\wedge S^V \xra{1_{X^*}\wedge\eta}
	X^*\wedge X\wedge &X^* \xra{\epsilon\wedge1_{X^*}}
	X^*\wedge S^V	\end{split}\end{equation}
are based $\TT$--homotopic to the factor exchange map antipodal map on~$S^V$.
\end{definition}

One can show that if~$X^*_i$ are~$V_i$--duals for the same~$X$, then~$X^*_1\wedge S^{V_2}$ and~$X^*_2\wedge S^{V_1}$ become isomorphic in the stable homotopy category (cf.~\cite{May_equivariant_homotopy_1996}*{Chs.~XVI.7--8}).
Put differently, $V$--duals -- if they exist -- are stably unique in the above sense.
As for existence, we record the following:

\begin{lemma}\label{T:type SWF duals}
Let~$X$ be a $\TT$--space of type SWF.
Then there exists another $\TT$--space of type~SWF and a $\TT$--representation of the form~$V\cong\R^k\oplus\C^m$ such that~$X$ and~$X^*$ are $V$--dual.
\end{lemma}
\begin{proof}
This can be proved by replacing~$X$ with a finite $\TT$--complex and building~$X^*$ and~$V$ cell by cell.
The key is to identify suitable duals for the $\TT$--spaces~$\TT_+$ and~$(\TT/\TT)_+\cong S^0$ corresponding to free orbits and fixed points, respectively.
Now $S^0$ is trivially $0$--dual to itself.
For~$\TT_+$ we uses the canonical embedding $\TT\hra \C$ and whose normal bundle~$\nu$ is canonically trivialized.
Atiyah duality~\cite{May_equivariant_homotopy_1996}*{Ch.~XVI, Thm.~8.1} then shows that~$\TT_+$ and~$T\nu\cong \TT_+\wedge S^1$ are $\C$--dual.
\end{proof}

For later reference, we record the following fact about the duality maps for $\TT$--spaces of type SWF.

\begin{lemma}\label{T:duality maps type SWF}
Let~$X$ and~$X^*$ be $V$--dual $\TT$--spaces with~$V\cong\R^k\oplus\C^m$ with duality maps~$\eta$ and~$\epsilon$.
If~$X^\TT\simeq S^\ell$, then~$(X^*)^\TT\simeq S^{k-\ell}$.
Moreover, the fixed point maps~$\eta^\TT$ and~$\epsilon^\TT$ are homotopy equivalences.
\end{lemma}
\begin{proof}
The compositions in \eqref{eq:duality maps} are $\TT$--homotopy equivalences by assumption.
Passing to fixed points gives a homotopy equivalence
	\begin{equation}\label{eq:duality maps fixed}
	S^k\wedge X^\TT \xra{\eta^\TT\wedge1}
	X^\TT\wedge (X^*)^\TT\wedge X^\TT \xra{1\wedge\epsilon^\TT}
	X^\TT\wedge S^k.
	\end{equation}
Since~$X^\TT\simeq S^\ell$ and~$(X^*)^\TT\simeq S^{\ell^*}$ for some~$\ell,\ell^*\ge0$, we must have~$k=\ell+\ell^*$ so that $\eta^\TT\wedge1$ and~$1\wedge\epsilon^\TT$ are maps between spheres of the same dimensions.
After fixing orientations, we have
	\begin{equation}
	\pm1 = \deg(1\wedge\epsilon^\TT)\deg(\eta^\TT\wedge1) = \deg(\epsilon^\TT)\deg(\eta^\TT)
	\end{equation}
which shows that~$\eta^\TT$ and~$\epsilon^\TT$ have degree~$\pm1$.
\end{proof}

As an immediate consequence, we get the following:

\begin{corollary}\label{T:SW duality for maps}
Let~$f\colon X_1\to X_2$ be a $\TT$--map between $\TT$--spaces of type SWF with~$X_1^\TT\simeq X_2^\TT$.
Further, let $X_j^*$ be~$V_j$--dual to~$X_j$ with~$V_j\cong\R^{k_j}\oplus\C^{m_j}$ for~$j=1,2$.
Consider the composition
	\begin{equation*}
	f^*\colon 
	X_2^*\wedge S^{V_1} \xra{1\wedge\eta_1} 
	X_2^*\wedge X_1 \wedge X_1^* \xra{1\wedge f\wedge1}
	X_2^*\wedge X_2 \wedge X_1^* \xra{\epsilon_2\wedge 1}
	S^{V_2}\wedge X_1^*.
	\end{equation*}
Then~$(f^*)^\TT$ is a map between spheres of the same dimensions and, after fixing orientations, has the same degree as~$f^\TT$ up to sign.
\end{corollary}

Another standard consequence of $V$--duality is that the role of homology and cohomology is interchanged.
For the theories represented by the $\TT$--spectra $F$, $C$ and~$T$ this means the following.

\begin{lemma}\label{T:V-duality vs co-homology}
Suppose that~$X$ and~$X^*$ are $V$--dual and write~$|V|=\dim(V)$.
There are natural isomorphisms that make the following diagram commute for every~$n\in\Z$:
	\begin{equation}\label{eq:V-duality diagram}\begin{tikzcd}
	F^nX \rar{\mu^n} \dar{\cong} & 
		C^nX \rar{\lambda^n} \dar{\cong} & 
			T^nX \rar{\delta^n} \dar{\cong} &
				F^{n+1}X \dar{\cong} \\ 
	F_{|V|-n}X^* \rar{\mu_{|V|-n}} & 
		C_{|V|-n}X^* \rar{\lambda_{|V|-n}} & 
			T_{|V|-n}X^* \rar{\delta_{|V|-n}} &
				F_{|V|-n-1}X^*
	\end{tikzcd}\end{equation}
\end{lemma}

This suggests that we should also consider homological $h$--invariants
\begin{definition}\label{D:h-invariants top hom}
Given a $\TT$--space of type SWF~$X$ and~$\kk$ a principal ideal domain, we define the \emph{homological $h$--invariants} as
	\begin{align*}
	h_w(X;\kk) 
	&= \max\Set{k\in\Z}{\text{$C_{\ell+2k}X\xra{\lambda_*} T_{\ell+2k}X$ is non-zero}} 
	\\[1em]
	h_s(X;\kk) 
	&= \max\Set{k\in\Z}{\text{$C_{\ell+2k}X\xra{\lambda_*} T_{\ell+2k}X$ is surjective}}. 	\end{align*}
For fields we write~$h_\F(X)=h_{s/w}(X;\F)$ and~$h_p(X)=h_{\F_p}(X)$.
\end{definition}

As a consequence of \cref{T:V-duality vs co-homology}, we get an analogue of~\cref{T:monopole h-inv hom vs coh}:

\begin{corollary}\label{T:top h-inv hom vs coh}
Suppose that~$X$ and~$X^*$ are $\TT$--spaces of type SWF which are $V$--dual with~$V\cong\R^k\oplus\C^m$.
Then for every commutative ring~$\kk$ we have
	\begin{equation}
	h^{s/w}(X;\kk) = m-h_{s/w}(X^*).
	\end{equation}
\end{corollary}
\begin{proof}
It suffices to note that if~$X^\TT\simeq S^\ell$, then~$(X^*)^\TT\simeq S^{\ell*}$ with~$\ell+\ell^*=k$.
Thus for~$n=\ell+2h$ we find $|V|-n = \ell^*+2(m-h)$.
The claim now follows from \cref{T:V-duality vs co-homology} and the definitions.
\end{proof}

\subsection{Kronecker duality}
The second notion of duality relates the homological and cohomological $h$--invariants of the same space using Kronecker pairings derived from the multiplicative structures of the spectra~$F$, $C$, and~$T$ explained in \cref{R:ring spectra}.
We first note that the isomorphisms in~\eqref{eq:representing Borel} give a pairing
	\begin{equation}\label{eq:Kronecker ordinary}
	C^nX\otimes_\kk F_{n+1}X
	\cong \tH^n(E\TT_+\wedge\TT X) \otimes_\kk \tH_n(E\TT_+\wedge\TT X)
	\to\kk
	\end{equation}
derived from the ordinary Kronecker pairing.
Given any $\kk$--module, we write $M^*=\Hom_\kk(M,\kk)$ for the dual $\kk$--module.
The pairing above thus gives a $\kk$--linear map $C^nX\to(F_{n+1}X)^*$.
Making more systematic use of the multiplicative structures of~$F$, $C$, and~$T$, we obtain the following:
\begin{lemma}\label{T:Kronecker}
Let~$X$ be a $\TT$--space of type SWF.
The multiplicative structures of~$F$, $C$, and~$T$ give rise to a commutative diagram
	\begin{equation}\label{eq:Kronecker diagram}\begin{tikzcd}
		C^nX \rar{\lambda^n} \dar & 
			T^nX \dar{\cong} \\
		(F_{n+1}X)^* \rar{\delta_{n+2}^*} & 
			(T_{n+2}X)^* 
	\end{tikzcd}\end{equation}
where the map $C^nX\to(F_{n+1}X)^*$ corresponds to the pairing in~\eqref{eq:Kronecker ordinary}.
In particular, it is an isomorphism for field coefficients.
The map~$T^nX\to(T_{n+2}X)^*$ is always an isomorphism.
\end{lemma}
\begin{proof}
The multiplicative structures give a commutative diagram
	\begin{equation*}\begin{tikzcd}
	C^n\otimes_\kk F_{n+1} \ar[d] &
		C^nX\otimes_\kk T_{n+2}X \ar[l,"1\otimes\delta_{n+2}"'] \ar[r,"\lambda^n\otimes1"] \ar[dr] &
			T^n\otimes_\kk T_{n+2} \ar[d] \\
	F_1S^0 && T_2S^0 \ar[ll,"\delta_2"',"\cong"]
	\end{tikzcd}\end{equation*}
where the unlabeled maps are Kronecker pairings.
The map~$\delta_2$ is an isomorphism, because by~\eqref{eq:representing Borel} the groups $C_kS^0\cong C^{-k}S^0\cong \BC^{-k}(S^0)$ vanish.
Similarly, we find $F_1S^0 \cong \BH_0(S^0;\kk)\cong\kk$.
From this we can derive the commutative diagram~\eqref{eq:Kronecker diagram}.
The asserted relation of $C^nX\to(F_{n+1}X)^*$ and~\eqref{eq:Kronecker ordinary} follows from writing out the definitions, which is somewhat tedious but ultimately straightforward.
The ordinary universal coefficient theorem shows that the map $C^nX\to(F_{n+1}X)^*$ is surjective, and an isomorphisms if~$\kk$ is a field or, more generally, if~$F_{n+1}X\cong \BH_n(X)$ is torsion free over~$\kk$.
As for the map $T^nX\to(T_{n+2}X)^*$, we may assume that~$n\ge0$.
Indeed, multiplication by~$u^n$ gives a commutative diagram
	\begin{equation*}\begin{tikzcd}
	T^nX \dar & T^{-n}X \ar[l,"\cong","u^n"']\dar \\
	(T_{n+2}X)^* & (T_{-n+2}X)^* \ar[l,"\cong","(u^n)^*"']
	\end{tikzcd}\end{equation*}
with vertical maps induced by the Kronecker pairings, and by \cref{T:Localization} multiplication by~$u^n$ is invertible for the Tate theories.
Moreover, for~$n\ge0$ we get isomorphisms
	\begin{equation}
	T^nX\cong T^nX^\TT \cong C^nX^\TT \cong \BC^n(S^\ell),
	\end{equation}
and similarly
	\begin{equation}
	T_{n+2}X \cong T_{n+2}X^\TT \cong F_{n+1}X^\TT \cong \BH_n(S^\ell).
	\end{equation}
The map~$T^nX\to(T_{n+2}X)^*$ is taken to the map~$C^nX^\TT\to(F_{n+1}X^\TT)^*$ which is again an isomorphism by the universal coefficient theorem, since~$\BH_n(S^\ell)$ is torsion free.
\end{proof}

We can use \cref{T:Kronecker} to compare the homological and cohomological \mbox{$h$--invariants} of the same space.
This is particularly simple for field coefficients.

\begin{proposition}\label{T:duality SWF type}
Let~$X$ be a $\TT$--space of type SWF.
If $\F$ a field, then
	\begin{equation}
	h_\F(X)=h^\F(X).
	\end{equation}
If~$\kk$ is a principal ideal domain, then
	\begin{equation}\label{eq:h-invariants inequalities}
	h_s(X;\kk)\le h_w(X;\kk)=h^w(X;\kk)\le h^s(X;\kk).
	\end{equation}
\end{proposition}
\begin{proof}
We first discuss the case with coefficients in~$\F$.
We can express as the $h$--invariants as
	\begin{equation}\label{eq:h via exactness}
	h^\F(X) = \min\Set{k}{\lambda^{\ell+2k}\ne0} \AND 
	h_\F(X) = \min\Set{k}{\delta_{\ell+2k+2}\ne0}
	\end{equation}
using \cref{T:h via Tate} for~$h^\F(X)$ and the exactness of $C_*X\xra{\lambda_*}T_*X \xra{\delta_*} F_{*-1}X$ for~$h_\F(X)$.
Now \cref{T:Kronecker} shows that
	\begin{equation}
	\lambda^n\ne0 \Leftrightarrow
	\delta_{n+2}^* \Leftrightarrow
	\delta_{n+2}\ne0
	\end{equation}
which together with~\eqref{eq:h via exactness} implies~$h^\F(X)=h_\F(X)$.
In the case of a principal ideal domain~$\kk$, we can only draw the weaker conclusion
	\begin{equation}
	\lambda^n \ne 0 \Leftrightarrow 
	\delta_{n+2}^* \ne 0 \Rightarrow
	\delta_{n+2} \ne 0.
	\end{equation}
Using the obvious analogue of~\eqref{eq:h via exactness} we get~$h^w(X;\kk)\ge h_s(X;\kk)$.
The equality~$h^w(X;\kk)=h_w(X;\kk)$ follows from the field coefficient case after passing to the field of fractions and invoking~\cref{T:basic inequalities h-invariant}.
The remaining inequality in~\eqref{eq:h-invariants inequalities} is trivial.
\end{proof}

Combining both notions of duality, we arrive at the following:

\begin{corollary}\label{T:duality top}
Suppose that~$X$ and~$X^*$ are $\TT$--spaces of type SWF which are $V$--dual with~$V\cong\R^k\oplus\C^m$.
Then for every field~$\F$ we have
	\begin{equation}
	h^\F(X) = m-h^\F(X^*) \AND 
	h_\F(X) = m-h_\F(X^*).
	\end{equation}
For a principal ideal domain~$\kk$ we have
	\begin{equation}
	h_s(X;\kk)+h_s(X^*;\kk) \le m \le h^s(X;\kk)+h^s(X^*;\kk)
	\end{equation}
\end{corollary}
\begin{proof}
Immediate from \cref{T:top h-inv hom vs coh,T:duality SWF type}.
\end{proof}

\subsection{Additivity}

In this section, we investigate the behavior of the topological $h$--invariants under smash products.
Clearly, the class of $\TT$--spaces of type SWF is closed under this operation.
We first consider a special case, which can be interpreted as a stability property.

\begin{lemma}\label{T:stability}
If~$X$ is a $\TT$--space of type~SWF, then
	\begin{equation}\label{eq:stability}
	h^{s/w}(X\wedge S^{\ell,m};\kk) = h^{s/w}(X;\kk)+m
	\end{equation}
for all choices of~$\kk$ and all~$\ell,m\ge0$.
\end{lemma}
\begin{proof}
Suppose that~$X^\TT\simeq S^k$.
Then $(X\wedge S^{\ell,m})^\TT\cong X^\TT\wedge S^{\ell}\simeq S^{k+\ell}$, and the $h$--invariants can be compared using the commutative diagram
	\begin{equation*}\begin{tikzcd}
	\BC^{k+\ell+2n}(X\wedge S^{\ell,m}) \dar{i^*_{X\wedge S^{\ell,m}}} \rar{\cong} &
	\BC^{k+2n}(X\wedge S^{0,m}) \dar{i^*_{X\wedge S^{0,m}}} &
	\BC^{k+2(n-m)}(X) \dar{i^*_{X}} \lar{\cong}
	\\
	\BC^{k+\ell+2n}(X^\TT\wedge S^{\ell}) \rar{\cong} &
	\BC^{k+2n}(X^\TT) &
	\BC^{k+2(n-m)}(X^\TT) \ar[l,"\cong","u^m"']
	\end{tikzcd}\end{equation*}
which is provided by the Thom and Euler class interpretation of the suspension isomorphisms and~$u$ in~(B\ref{fact:module structure}) and~(B\ref{fact: suspension isomorphisms}).
\end{proof}

The situation for two arbitrary $\TT$--spaces of type SWF is more complicated and we can only prove additivity for field coefficients.

\begin{proposition}\label{T:additivity top}
If~$X_1$ and~$X_2$ are $\TT$--spaces of type SWF, then so is~$X_1\wedge X_2$ and for every field~$\F$ we have
	\begin{equation}\label{eq:additivity top h-inv}
	h^\F(X_1\wedge X_2) = h^\F(X_1)+h^\F(X_2).
	\end{equation}
Moreover, for every principal ideal domain~$\kk$ we have
	\begin{equation}\label{eq:subadditivity top h-inv}
	h^s(X_1\wedge X_2;\kk) \le h^s(X_1;\kk)+h^s(X_2;\kk)
	\end{equation}
and the inequality can be strict, in general.
\end{proposition}

\begin{proof}
We first establish the following sub-additivity property~\eqref{eq:subadditivity top h-inv}.
Since Borel cohomology is represented by a ring $\TT$--spectrum, there are natural products
	\begin{equation}
	\BC^m(X_1)\otimes_\kk\BC^n(X_2) \xra{\cup} \BC^{m+n}(X_1\wedge X_2)
	\end{equation}
which are compatible with the $\kk[u]$--module structure in the sense that
	\begin{equation}
	x_1\cup(ux_2)=(ux_1)\cup x_2 = u(x_1\cup x_2)
	\end{equation}
for all~$x_j\in \BC^*(X_j)$, $j=1,2$.
Now suppose that $X_j^\TT\simeq S^{\ell_j}$, $j=1,2$.
The products give isomorphisms
	\begin{equation}\label{eq:h-invariants sub-additive}
	\BC^m(S^{\ell_1})\otimes_\kk\BC^n(S^{\ell_2}) \overunderset{\cup}{\cong}{\lra} \BC^{m+n}(S^{\ell_1+\ell_2}).
	\end{equation}
We write~$i_j\colon X_j^\TT\hra X_j$ for the fixed point inclusions and implicitly identify $i_1\wedge i_2$ with the fixed point inclusion of~$X_1\wedge X_2$.
For arbitrary $x_j\in\BC^{\ell_j+2h_j}(X_j)$ the naturality of the products gives
	\begin{equation}
	(i_1\wedge i_2)^*(x_1\cup x_2) = (i_1^*x_1)\cup(i_2^*x_2) \in \BC^{\ell_1+\ell_2+2(h_1+h_2)}(X_1\wedge X_2).
	\end{equation}
In particular, if~$i_1^*x_1$ and~$i_2^*x_2$ are both generators (resp.~non-zero), then so is $(i_1\wedge i_2)^*(x_1\cup x_2)$ which implies~\eqref{eq:subadditivity top h-inv}.

Now let~$\F$ be a field.
By \cref{T:type SWF duals} we can find $\TT$--spaces of type~SWF~$X_j^*$ and $\TT$--representations of the form~$V_j\cong \R^{k_j}\oplus\C^{m_j}$ such that~$X_j$ and~$X_j^*$ are $V_j$--dual.
One can easily check that $X_1^*\wedge X_2^*$ and~$X_1\wedge X_2$ are $V_1\oplus V_2$--dual.
Combining the sub-additivity property~\eqref{eq:subadditivity top h-inv} with \cref{T:duality top}, we obtain the super-additivity property
	\begin{equation}\begin{split}
	h^\F(X_1\wedge X_2)
	& = m_1+m_2 -  h^\F(X_1^*\wedge X_2^*) \\
	& \ge m_1+m_2 -  h^\F(X_1^*)-h^\F(X_2^*) \\
	& = h^\F(X_1)+h^\F(X_2).
	\end{split}\end{equation}
and altogether the additivity property~\eqref{eq:additivity top h-inv}.

Lastly, we show the~\eqref{eq:subadditivity top h-inv} can be strict.
Let~$a,b,c\in\N$ be distinct primes.
Then $X_{a,b}$ and~$X_{a,c}$ from \cref{eg:example for h-difference} satisfy
	\begin{equation*}
	h_p(X_{a,b}) = 
		\begin{cases}
		h+1,& p=b\\
		h, & \text{else}
		\end{cases}
	\AND
	h_p(X_{a,c}) = 
		\begin{cases}
		h+1,& p=c\\
		h, & \text{else.}
		\end{cases}
	\end{equation*}
On the one hand, \cref{T:basic inequalities h-invariant}(iv) and~\eqref{eq:additivity top h-inv} give
	\begin{equation*}\begin{split}
	h^s(X_{a,b}\wedge X_{a,c};\Z)
	&= \max_p h_p(X_{a,b}\wedge X_{a,c}) \\
	&= \max_p \big( h_p(X_{a,b}) + h_p(X_{a,c}) \big)
	= 2h+1.
	\end{split}\end{equation*}
On the other hand, \cref{T:basic inequalities h-invariant}(iv) applied individually shows
	\begin{equation*}
	h^s(X_{a,b};\Z) + h^s(X_{a,c};\Z)
	= 2h+2.\qedhere
	\end{equation*}
\end{proof}

\subsection{Monotonicity}

Let~$f\colon X\to X'$ be a $\TT$-map between $\TT$--spaces of type~SWF.
Up to homotopy equivalence, the fixed point map~$f^\TT$ is a non-equivariant map between spheres of dimensions~$\ell$ and~$\ell'$, say.
If~$\ell=\ell'$, then the mapping degree of~$f^\TT$ is well-defined up to sign.
We have the following monotonicity properties for the $h$--invariants:

\begin{proposition}\label{T:monotonicity}
Let~$f\colon X\to X'$ be a $\TT$--map between $\TT$--spaces of type~SWF with~$X^\TT\simeq X'{}^\TT$.
\begin{enumerate}[(i)]
\item
If~$\deg(f^\TT)\ne0$, then $h^w(X;\kk)\le h^w(X';\kk)$ for all~$\kk$.
\item
If~$\deg(f^\TT)=\pm1$, then $h^s(X;\kk)\le h^s(X';\kk)$ for all~$\kk$.
\end{enumerate}
\end{proposition}
\begin{proof}
We consider the commutative diagram
	\begin{equation*}\begin{tikzcd}
	\BC^*(X') \ar[r,"f^*"] \ar[d,"i_{X'}^*"] &
		\BC^*(X) \ar[d,"i_{X}^*"] \\
	\BC^*(X'{}^\TT) \ar[r,"(f^\TT)^*"] &
		\BC^*(X^\TT).
	\end{tikzcd}\end{equation*}
After fixing identifications $X^\TT\simeq S^\ell\simeq X'{}^{\TT}$, we can consider~$f^\TT$ as a map $g\colon S^\ell\to S^\ell$.
The induced map on Borel cohomology is really the map $(\id\wedge_\TT g)^*$ on the ordinary cohomology of~$\ETwt S^\ell\cong \BT_+\wedge S^\ell$.
Under the indicated homeomorphism, we have $(\id\wedge_\TT g)^*=\id\otimes g^*$ which proves the claim.
\end{proof}

\begin{remark}\label{R:Kunneth approach}
Recall that we have proved the additivity property for field coefficients in \cref{T:additivity top} using duality in the form of~\cref{T:duality SWF type}.
There is an alternative approach to additivity based on a Künneth sequence
	\begin{equation*}\label{eq:Kunneth sequence fields}
	0 \to 
	\BC^*(X_1)\otimes_{\F[u]}\BC^*(X_2) \to 
	\BC^*(X_1\wedge X_2) \to 
	\Tor_{\F[u]}^{*+1}\big( \BC^*(X_1), \BC^*(X_2) \big) \to 0
	\end{equation*}
which can be derived from an Eilenberg--Moore spectral sequence for the fibrations~$\ET\times_\TT X_j\to\BT$. 
This has the advantage of bypassing the Tate and coBorel theories and the duality arguments.
However, one has to control the $\Tor$--term instead. 

In fact, the duality and additivity properties are equivalent.
Assuming that \cref{T:additivity top} has been proved independently, one can deduce \cref{T:duality top} using the monotonicity property in \cref{T:monotonicity} in conjunction with~\cref{T:duality maps type SWF}.
Together with \cref{T:V-duality vs co-homology} one recovers \cref{T:duality SWF type}.
\end{remark}

\section{The relation to Seiberg--Witten theory}
\label{ch:relation to SW}

In this final section we discuss the relation of monopole $h$--invariants and their topological analogues and use this to prove \cref{T:Froyshov intro,T:additivity etc intro,T:coeff dep intro}.
The two types of invariants are linked by the work of Lidman and Manolescu in~\cite{LidmanManolescu_equivalance_2018} which recovers the monopole Floer homology groups from the Seiberg--Witten--Floer homotopy types.

\subsection{Comparing the h--invariants}
Let $Y$ be a closed, connected \spinc 3--manifold with~$b_1(Y)$.
Recall from the introduction that the Seiberg--Witten--Floer homotopy type~$SWF(Y)$ is represented by triples~$(X,\ell,n)$ where~$X$ is a $\TT$--space of type SWF with~$X^\TT\simeq S^\ell$ and $n\in \Z+n(Y)=-\frac12\mf q^\mathrm{ev}(Y)$.
We can compare the $h$--invariants using a variant of the Lidman--Manolescu isomorphism in~\eqref{eq:Lidman-Manolescu intro}.

\begin{theorem}[Lidman--Manolescu~\cite{LidmanManolescu_equivalance_2018}]\label{T:Lidman-Manolescu cohomology}
Let $Y$ be a closed, connected \spinc 3--manifold with~$b_1(Y)=0$.
Suppose that~$SWF(Y)$ is represented by~$(X,\ell,n)$ and let~$\nu = \ell+2n$. 
Then for every principal ideal domain~$\kk$ there is a commutative diagram of~$\kk[u]$
	\begin{equation}\label{eq:LM diagram cohomology}\begin{tikzcd}
	\HMto^*(Y;\kk) \rar{i^*} \dar{\cong} & 
		\HMbar^*(Y;\kk) \dar{\cong} \\
	\BC^{*+\nu}(X;\kk) \rar{\lambda^*} &
		\tBC^{*+\nu}(X;\kk)
	\end{tikzcd}\end{equation}
in which the vertical maps are isomorphisms and respect the gradings.
\end{theorem}
This statement does not appear explicitly in~\cite{LidmanManolescu_equivalance_2018}, but it can be deduced from the proofs of~\cite{LidmanManolescu_equivalance_2018}*{Thm.~1.2.1, Cor.~1.1.4}.
We briefly sketch the argument.
\begin{proof}
To see that the gradings in~\eqref{eq:LM diagram cohomology} make sense, recall that~$n\in\Z+n(Y)$ and $\mf q(Y)=\Z-2n(Y)$ so that~$q+\nu\in\Z$ for every~$q\in\mf q(Y)$.
The diagram can be extracted from the arguments in~\cite{LidmanManolescu_equivalance_2018}*{Ch.~14} as follows.
Lidman and Manolescu argue that for each~$N$ there is a way to realize~$\BH_{q+\nu}(X;\Z)$ in the range~$q+\nu\in[-N,N]$ as the homology of a chain complex of Morse--Floer type which is explicitly isomorphic to the part of a chain complex defining~$\HMto_*(Y;\Z)$ in this range.
Moreover, there are chain maps realizing the $u$--actions which correspond under the chain isomorphism in the range~$[-N+1,N-1]$.
Applying~$\Hom_\Z(\;,\kk)$ in each degree results in chain complexes computing~$\HMto^{q}(Y;\kk)$ and~$\BC^{q+\nu}(X;\kk)$ in a range.
Since~$N$ was arbitrary, we get isomorphisms
	\begin{equation}
	\HMto^q(Y;\kk)\cong\BC^{q+\nu}(X;\kk)
	\end{equation}
for all~$q\in\mf q(Y)$ which are compatible with the $u$--actions.
Since~$i^*$ and~$\lambda^*$ are both known to be localization maps, we also get isomorphisms
	\begin{equation}
	\HMbar_*(Y;\kk) 
	\cong u\inv \HMfrom_*(Y;\kk) 
	\cong u\inv \cBH_{*+\nu}(X;\kk)
	\cong \tBH_{*+\nu}(X;\kk)
	\end{equation}
completing the square in~\eqref{eq:LM diagram cohomology} such that it commutes.
\end{proof}
The relation between the monopole $h$--invariants and their topological analogues is now obvious.

\begin{corollary}\label{T:comparing h-invariants cohomology}
Let~$Y$ be a closed, connected \spinc 3--manifold with~$b_1(Y)=0$ and~$\kk$ a principal ideal domain.
If~$SWF(Y)$ is represented by~$(X,\ell,n)$, then
	\begin{equation}
	h^{s/w}(Y;\kk) = h^{s/w}(X;\kk) - n \AND 
	h_{s/w}(Y;\kk) = h_{s/w}(X;\kk) - n.
	\end{equation}
\end{corollary}
\begin{proof}
By \cref{T:top h-inv hom vs coh,T:monopole h-inv hom vs coh} it suffices to prove the cohomological version which follows immediately from \cref{T:Lidman-Manolescu cohomology,T:h via Tate}.
\end{proof}

\subsection{Duality and additivity}
\label{ch:additivity monopoles}

\cref{T:comparing h-invariants cohomology} allows to translate results about the topological $h$--invariants into results about monopole $h$--invariants.
We first address the duality and additivity properties of the monopole $h$--invariants.

\begin{theorem}\label{T:duality}
Let~$Y$ be a closed, connected \spinc 3--manifold with~$b_1(Y)=0$.
If~$\kk$ is a principal ideal domain, then 
	\begin{equation}\label{eq:duality monopole strong}
	h_s(Y,\kk)+h_s(-Y,\kk)\le 0 \le h^s(Y,\kk)+h^s(-Y,\kk).
	\end{equation}
If~$\F$ is a field, then 
	\begin{equation}\label{eq:Kronecker duality monopole}
	h_\F(Y)=h^\F(Y)
	\end{equation}
and thus
	\begin{equation}\label{eq:duality monopole weak}
	h_\F(-Y)=-h_\F(Y) \AND
	h^\F(-Y)=-h^\F(Y).
	\end{equation}
\end{theorem}
\begin{proof}
The equality in~\eqref{eq:Kronecker duality monopole} follows from~\cref{T:comparing h-invariants cohomology} and~\cref{T:duality SWF type} and together with~\eqref{eq:duality monopole strong} implies~\eqref{eq:duality monopole weak}.
By \cref{T:monopole h-inv hom vs coh} it suffices to prove the inequality~$h^s(Y,\kk)+h^s(-Y,\kk)\ge0$ in~\eqref{eq:duality monopole strong}.
It is proved in~\cite{Manolescu_gluing_theorem_2007} that $SWF(Y)$ and~$SWF(-Y)$ can be represented by triples $(X,\ell,n)$ and~$(X^*,\ell^*,n^*)$, respectively, such that~$X$ and~$X^*$ are $V$--dual with $V\cong \R^{\ell+\ell^*}\oplus\C^{n+n^*}$.
The inequality then follows from \cref{T:duality SWF type} and \cref{T:duality top} using \cref{T:comparing h-invariants cohomology}.
\end{proof}

\begin{theorem}\label{T:additivity}
Let~$Y$ and~$Y'$ be closed, connected \spinc 3--manifolds with $b_1(Y)=b_1(Y')=0$.
If~$\F$ is a field, then
	\begin{equation}\label{eq:additivity field hom}
	h_\F(Y\#Y')=h_\F(Y)+h_\F(Y)
	\end{equation}
If~$\kk$ is a principal ideal domain, we have inequalities
	\begin{equation}\label{eq:additivity mon strong}\begin{split}
	&h_s(Y\#Y';\kk)\ge h_s(Y;\kk)+h_s(Y';\kk) \AND\\
	&h^s(Y\#Y';\kk)\le h^s(Y;\kk)+h^s(Y';\kk)
	\end{split}\end{equation}
\end{theorem}
\begin{proof}
Again by \cref{T:monopole h-inv hom vs coh}, it suffices to prove the cohomological statement in~\eqref{eq:additivity mon strong}.
Moreover, since we established~$h_\F=h^\F$ in \cref{T:duality}, it also suffices to prove the cohomological analogue of \eqref{eq:additivity field hom}.

Now, if~$SWF(Y)$ and~$SWF(Y')$ are represented by~$(X,\ell,n)$ and~$(X',\ell,n')$, respectively, then $SWF(Y\#Y')$ is represented by~$(X\wedge X',\ell+\ell',n+n')$ according to~\cite{DaiSasahiraStoffregen_lattice_vs_SWF_arxiv_v1_2023}*{Thm.~3.19}.
The claims now follow from \cref{T:additivity top} and \cref{T:comparing h-invariants cohomology}.
\end{proof}
Based on \cref{T:additivity top}, we expect the inequalities in \eqref{eq:additivity mon strong} to be strict.
However, we are not aware of any examples with strict inequality.

\subsection{\Froyshov inequalities and monotonicity}
\label{ch:Froyshov}

Every closed \spinc 3--manifold~$Y$ bounds some compact \spinc 4--manifold~$W$ which we may assume to be connected.
We refer to such a~$W$ as a \emph{\spinc filling} of~$Y$.
Assuming that~$b_1(Y)=0$, the cup product and the orientation of~$W$ give rise to a non-degenerate rational valued pairing on~$H^2(W;\Q)\cong H^2(W,Y;\Q)$.
Let~$b_2^\pm(W)$ be the maximal dimension of positive and negative subspaces \wrt that pairing and~$\sigma(W)=b_2^+(W)-b_2^-(W)$ its signature.
The metric and reference connection on~$Y$ maybe extended to~$W$.
The associated \spinc Dirac operator~$D^+\colon\Gamma(S_W^+)\to\Gamma(S_W^-)$ with Atiyah--Patodi--Singer boundary conditions is a Fredholm operator and, according to~\cite{Manolescu_SWF_spectra_2003}*{p.~909}, its index can be expressed as 
	\begin{equation}\label{eq:APS index}
	\ind D^+ = \frac18\big(c_1(S_W^+)^2-\sigma(W)\big) + n(Y).
	\end{equation}
The monopole $h$--invariants of~$Y$ are known to put constraints on the possible fillings~$W$ with~$b_2^+(W)=0$ (cf.~\cite{KronheimerMrowka_book_2007}*{Thm.~39.1.4}, \cite{BehrensGolla_twisted_2018}*{Ch.~5}).
We can derive the same constraints from \cref{T:monotonicity}.

\begin{theorem}\label{T:Froyshov main}
Let~$\kk$ be any principal ideal domain.
Then~$h_{s/w}(S^3;\kk)=0$ and if~$Y$ is a closed, connected 3--manifold with~$b_1(Y)$ and~$W$ is \spinc filling of~$Y$ with~$b_2^+(W)=0$, then
	\begin{equation}
	\frac18\big(c_1(S_W^+)^2+b_2(W)\big) \le h_{s/w}(Y;\kk).
	\end{equation}
The same statements hold with~$h^{s/w}(Y;\kk)$ in place of~$h_{s/w}(Y;\kk)$.
\end{theorem}
\begin{proof}
According to~\cites{Manolescu_SWF_spectra_2003,Khandhawit_gauge_slice_2015}, a finite dimensional approximation to the Seiberg--Witten map on~$W$ gives rise to a $\TT$--map
	\begin{equation}
	f_W\colon S^{\ell-b_2^+(W),n+d(W)}\to X,\quad d(W)=\frac18\big(c_1^2(W)-\sigma(W)\big)
	\end{equation}
Moreover, if~$b_2^+(W)=0$, then $-\sigma(W)=b_2(W)$ and the fixed point map $f_W^\TT\colon S^\ell\to X^\TT$ is a homotopy equivalence (see~\cite{Manolescu_SWF_spectra_2003}*{Ch.~10}).
\cref{T:monotonicity} therefore gives 
	\begin{equation}
	n+d(W)=S^{\ell,n+d(W)} \le h^{s/w}(X;\kk)
	\end{equation}
and we may rewrite this using \cref{T:comparing h-invariants cohomology} as
	\begin{equation}
	\frac18\big(c_1(W)^2+b_2(W)\big) \le h^{s/w}(X;\kk)-n = h^{s/w}(Y).
	\end{equation}
If~$(X^*,\ell^*,n^*)$ represents~$SWF(-Y)$ and~$X^*$ is $(\R^{\ell+\ell*}\oplus\C^{n+n^*})$--dual to~$X$, then by \cref{T:SW duality for maps} the composition
	\begin{equation}
	f_W^*\colon X^*\wedge S^{\ell,n+d(W)} \xra{\id\wedge f_W} X^*\wedge X \xra{\epsilon} S^{\ell^*,n^*}\wedge S^{\ell,n}
	\end{equation}
also has the property that~$(f_W^*)^\TT$ is a homotopy equivalence.
As above, we get
	\begin{equation}
	h^{s/w}(X^*)+n+d(W) \le n+n^*
	\end{equation}
which is equivalent to
	\begin{equation}
	\frac18\big(c_1(W)^2+b_2(W)\big) \le -h^{s/w}(-Y;\kk)=h_{s/w}(Y;\kk)
	\end{equation}
using \cref{T:monopole h-inv hom vs coh} in the last step.
\end{proof}

We now pass to coefficients in the prime fields~$\F_p$.
The monotonicity property in~\cref{T:additivity etc intro} is now a formal consequence.
\begin{corollary}\label{T:monotonicity mon}
Let~$W$ be a 4--dimensional \spinc cobordism from~$Y$ to~$Y'$ with $b_1(Y)=b_1(Y')=0$ and~$b_2^+(W)=0$.
Then
	\begin{equation}
	h_p(Y) + \frac18\big(c_1(W)^2+b_2(W)\big) \le h_p(Y').
	\end{equation}
\end{corollary}
\begin{proof}
Consider~$W$ as a \spinc filling of~$-Y\amalg Y'$.
By removing an open neighborhood of an embedded arc from~$Y$ to~$Y'$ we obtain a \spinc filling~$W^\circ$ of~$(-Y)\#Y'$ with the same $b_2^\pm$ and~$c_1^2$ as~$W$.
Combining \cref{T:Froyshov main,T:additivity,T:duality} gives
	\begin{equation*}
	\frac18\big(c_1(W)^2+b_2(W)\big) \le h_p((-Y)\#Y') = -h_p(Y)+h_p(Y').\qedhere
	\end{equation*}
\end{proof}

\subsection{Dependence on the choice of coefficients}

Lastly, we address \cref{T:coeff dep intro}.
We prove a more general statement involving both the homological and cohomological $h$--invariants.

\begin{theorem}\label{T:coeff dep monopoles}
Let~$Y$ be a closed, connected \spinc 3--manifold with~$b_1(Y)=0$ and~$\kk$ a principal ideal domain of characteristic~$p$.
Then we have equalities
	\begin{equation}\label{eq:weak monopole relations}
	h_w(Y;\kk) = 
	h_p(Y) = 
	h^p(Y) = 
	h^w(Y;\kk).
	\end{equation}
and inequalities
	\begin{equation}\label{eq:strong monopole relations k}
	h_s(Y;\Z)\le h_s(Y;\kk) \le h^s(Y;\kk)\le h^s(Y;\Z)
	\end{equation}
as well as
	\begin{equation}\label{eq:strong monopole relations Z}
	h_s(Y;\Z) = \min_p h_p(Y) \le \max_p h_p(Y)=h^s(Y;\Z).
	\end{equation}
\end{theorem}
\begin{proof}
In the light of \cref{T:comparing h-invariants cohomology}, the inequalities and equalities involving only the cohomological $h$--invariants are immediate from \cref{T:basic inequalities h-invariant}.
The parts involving only the homological parts then follow from \cref{T:monopole h-inv hom vs coh}.
The relations between the homological and cohomological $h$--invariants follows from \cref{T:duality SWF type}.
\end{proof}
Based on \cref{eg:example for h-difference}, we expect that the invariants~$h_p(Y)$ generally depend on the choice of~$p$, but we are unaware of any examples that exhibit such dependence.

\subsection{Comparison with previous definitions}
\label{ch:precursors}
As mentioned in the introduction, the $h$--invariants discussed here have many precursors dating back to \Froyshov's article~\cite{Froyshov_SW_boundary_1996} from 1996.

\begin{enumerate}[(1)]

\item 
In \cite{Manolescu_triangulaion_JAMS_2016}*{Ch.~2.6} Manolescu defines invariants~$d_p(X)$ of spaces of type SWF which are related to ours by $d_p(X) = \ell+2h^p(X)$ by \cref{T:h via Tate}.
Later in~\cite{Manolescu_triangulaion_JAMS_2016}*{Ch.~3.7} he defines invariants~$\delta_p(Y)$ of 3--manifolds which agree with~$h^p(Y)$ via \cref{T:comparing h-invariants cohomology} and thus with~$h_p(Y)$ by~\cref{T:duality}.

\item 
The invariant~$Fr(Y)$ defined in~\cite{KronheimerMrowkaOzsvathSzabo_lens_space_surgeries_2007}*{Def.~2.11} is essentially a special case of \cref{D:monopole h-invariants}. 
After conventions for normalization, we have
$$Fr(Y)=2h^w(Y;\F_2)=2h_2(Y).$$

\item 
Similar, the invariant~$h(Y)$ defined in~\cite{KronheimerMrowka_book_2007}*{Def.~39.1.1} turns out to be
	\begin{equation}
	h(Y)= -h^s(Y;\R)=-h_0(Y).
	\end{equation}

\item\label{correction terms}
The \emph{correction terms} $d(Y)$ from~\cite{OzsvathSzabo_4mfs_gradings_2003}*{Def.~4.1} corresponds to
	\begin{equation}
	d(Y)=2h^w(Y;\Z)=2h_0(Y).
	\end{equation}
This follows from the isomorphisms between Heegaard--Floer and monopole Floer homology constructed in~\cites{KutluhanLeeTaubes_HM_vs_HF_I-V_2020,ColinGhigginiHonda_HF_vs_ECH_announcement_2011}.

\item
Similarly, the generalized correction terms~$\underline{d}_p(Y)$ from~\cite{BehrensGolla_twisted_2018}*{Def.~3.1} specialize to~$2h_p(Y)$ when~$b_1(Y)=0$.

\item 
In~\cite{Froyshov_monopole_homology_2010} \Froyshov defines invariants similar to~$h^\F(Y)$ where~$\F$ is a field (cf.~\cite{Froyshov_monopole_homology_2010}*{Thms.~1~\&~2}) and notes that his invariants depend only on the characteristic of~$\F$ (see \cite{Froyshov_monopole_homology_2010}*{p.~569}).
Since he uses an alternative construction of monopole Floer cohomology, it is not entirely clear how his $h$--invariants relate to the ones considered in this paper.
However, the inequality in \cite{Froyshov_monopole_homology_2010}*{Thm.~4} suggests that his $h$--invariants agree with~$-h_p(Y)$.

\item 
The earliest precursor is an invariant~$\gamma(Y)$ defined by \Froyshov in~\cite{Froyshov_SW_boundary_1996}*{Ch.~3}.
Since the definition predates that of monopole Floer homology, the relation to the $h$--invariants is not entirely clear. 
However, if~$Y$ admits a metric of positive scalar curvature, then one can show that~$SWF(Y)$ is represented by~$(S^0,0,n(Y))$ and~$\gamma(Y)=8n(Y)=8h_p(Y)$ for all~$p$.

\item 
Lastly, $h^s(Y;\Z)$ was implicitly studied by Manolescu in~\cite{Manolescu_SWF_spectra_2003}.
More explicitly, using the notation surrounding~\cite{Manolescu_SWF_spectra_2003}*{eq.~(20), p.~927} one can prove that $h^s(Y;\Z)=\inf_{g,\nu}\big(-n(Y,c,g)+\min\Set{r\in\Z_+}{\gamma_r=0}\big)$.
We plan to get back to this in a forthcoming article.
\end{enumerate}


\bibliography{sbehrens}

\end{document}